\documentclass{amsart}

\usepackage{amsmath}
\usepackage{amssymb}
\usepackage{xypic}
\usepackage{enumerate}
\usepackage{epic}
\theoremstyle{plain}
\numberwithin{equation}{section}

        \newtheorem{thm}[equation]{Theorem}
        \newtheorem{cor}[equation]{Corollary}
        \newtheorem{lem}[equation]{Lemma}
        \newtheorem{prop}[equation]{Proposition}

\theoremstyle{definition}
        \newtheorem{defn}[equation]{Definition}

        \newtheorem{example}[equation]{Example}
        \newtheorem{rem}[equation]{Remark}

\renewcommand{\proof}{{\bf Proof:\ }}
\newcommand{\Endproof}{\hspace*{\fill} $\Box$ \vspace{1ex} \noindent }

\newcommand{\A}{\mathcal{A}}
\newcommand{\B}{\mathcal{B}}
\newcommand{\Scal}{\mathcal{S}}
\newcommand{\iso}{\xrightarrow{\sim}}
\newcommand{\rhob}{\bar{\rho}}
\newcommand{\PGL}{{\rm PGL}}
\newcommand{\ZZ}{\mathbb{Z}}
\newcommand{\proj}{\mathbb{P}}
\newcommand{\LL}{\mathbb{L}}
\newcommand{\nats}{\mathbb{N}}
\newcommand{\ints}{\mathbb{Z}}
\newcommand{\rats}{\mathbb{Q}}
\newcommand{\OO}{\mathcal{O}}
\newcommand{\Aut}{{\rm Aut}}
\newcommand{\Gal}{{\rm Gal}}

\newcommand{\Spec}{{\rm Spec\,}}
\newcommand{\Frac}{{\rm Frac}}
\newcommand{\lcm}{{\rm lcm}}
\newcommand{\ol}[1]{\overline{#1}}
\newcommand{\mc}[1]{\mathcal{#1}}
\newcommand{\newmatrix}[4]{
    \left( \begin{array}{cc} #1 & #2 \\  #3 &
    #4 \end{array} \right)}
\newcommand{\Xb}{\bar{X}}
\newcommand{\Yb}{\bar{Y}}
\newcommand{\Vb}{\bar{V}}
\newcommand{\xb}{\bar{x}}
\newcommand{\zb}{\bar{z}}

\begin{document}

\title[Explicit resolution of weak wild quotient singularities]{Explicit resolution of weak wild quotient singularities on
arithmetic surfaces}


%    author one information
\author{Andrew Obus}
\address{Baruch College}
\curraddr{1 Bernard Baruch Way.  New York, NY 10010}
\email{andrewobus@gmail.com}
\thanks{The first author was supported by NSF Grant DMS-1602054}

%    author two information
\author{Stefan Wewers}
\address{Universit\"at Ulm}
\curraddr{Helmholzstra{\ss}e}
\email{stefan.wewers@uni-ulm.de}
\thanks{}

%    \subjclass and \keywords are not used by JAG.

%\date{}

%    "Communicated by" -- provide editor's name; required.
%\commby{Takeshi Saito}

\begin{abstract}
A weak wild arithmetic quotient singularity arises from the quotient of  a smooth arithmetic surface by a finite group action, where the inertia group of a point on a closed characteristic $p$ fiber is a $p$-group acting with smallest possible ramification jump.  In this paper, we give complete explicit resolutions of these singularities using deformation theory and valuation theory, taking a more local perspective than previous work has taken.  Our descriptions answer several questions of Lorenzini.  Along the way, we give a valuation-theoretic criterion for a normal snc-model of $\proj^1$ over a discretely valued field to be regular. 
\end{abstract}

\maketitle

\section{Introduction}\label{Sintro}

A closed point $x$ on an integral normal scheme $\mc{X}$ is called a \emph{quotient singularity} if the local ring $A:=\OO_{\mc{X},x}$ can be written as $A=B^G$, where $B$ is a regular local ring and $G$ is a finite group of local automorphisms of $B$. The quotient singularity $x\in\mc{X}$ is called \emph{tame} if we can choose $B$ and $G$ such that the order of $G$ is prime to the residue characteristic of $B$. Otherwise, we call it a \emph{wild quotient singularity}. In this paper we will be exclusively concerned with the case where $\mc{X}$ has dimension $2$. 

We assume that there exists a {\rm desingularization} of $\mc{X}$, i.e.\ a birational and proper morphism $f:\tilde{\mc{X}}\to\mc{X}$ such that $\tilde{\mc{X}}$ is regular (this holds under very mild assumptions on $\mc{X}$, see \cite{Li:rs}). Then we may also assume that the desingularization $f:\tilde{\mc{X}}\to\mc{X}$ is minimal with the property that the exceptional divisor $E:=f^{-1}(x)$ is a reduced normal crossing divisor on $\tilde{\mc{X}}$. The \emph{resolution graph} of the singularity is the dual graph of $E$, enhanced by the self intersection numbers of the irreducible components. One of the main questions motivating this work is the problem of classifying the resolution graphs of wild quotient surface singularities. As this question remains wide open, a more modest goal is to systematically produce explicit examples of such singularities with interesting resolution graphs.  

There is an extensive literature on quotient singularities on complex surfaces. For instance, Brieskorn \cite{brieskorn68} has classified the resolution graphs of such singularities.\footnote{It is expected that this classification carries over to the more general case of tame quotient singularities as defined above, but we are not aware of any general result in this direction (for the case of tame \emph{cyclic} quotient singularities, see \cite{CES:cf} and \cite{Steck18}).} It turns out the resolution graphs are either chains or have a unique node of valency $3$. Moreover, quotient singularities on complex surfaces are always rational (see e.g.\ \cite{brieskorn68}, Satz 1.7).  

Much less is known in the wild case. Lorenzini has shown (\cite{Lo:wqs}, Theorem 2.8) that the resolution graph of a quotient singularity in dimension $2$ is always a tree, and that the irreducible components of the exceptional fiber are smooth of genus $0$. However, wild quotient singularities in dimension $2$ need not be rational (counterexamples were first given by Artin in \cite{artin75}). Schr{\"o}er and Ito have shown (see \cite{itoschroeer}, Corollary 2.2) that the resolution graph of a wild quotient singularity has at least one node. In a series of papers (\cite{Lo:wqs}, \cite{Lo:wm}, \cite{Lo:wq}) Lorenzini has studied certain wild quotient singularities which arise on integral models of curves over local fields. The present paper is motivated by this work and arose from an attempt to answer some of the questions posed therein.

In particular, we restrict in this paper to the case of \emph{weak wild} arithmetic quotient singularities, see Definition \ref{def:aqs}.  These are wild quotient singularities arising from integral models of curves over local fields where the group action on the special fiber is \emph{weakly ramified}, i.e., has smallest possible ramification jump.  These singularities appear, for example, on models of curves with potentially good ordinary reduction, as is studied in \cite{Lo:wm}.  One can think of these singularities as the ``mildest possible" wild quotient singularities, and they seem to be the most amenable to study.  For instance, 
Lorenzini showed in \cite{Lo:wq} that they are rational when they arise from products of curves in characteristic $p$.  We show, in fact, that every weak wild arithmetic quotient singularity is rational (Corollary \ref{Crationalsingularity}).

\subsection{Results and techniques}\label{Sresults}
In this paper, we give a complete, explicit description of the resolution of any weak wild arithmetic quotient singularity (Theorem \ref{Tresolution}, Corollary \ref{Cwwsingres}).  We show that the resolution graph is a tree with at most $e$ nodes when the singularity comes from a $(\ints/p)^e$-action on a smooth arithmetic surface.  Furthermore, we relate the multiplicities and self-intersection numbers of components of the special fiber of the resolution to the arithmetic complexity of a certain extension of local fields, along with a generator of that extension.  This answers (generalizations of) several questions of Lorenzini from \cite{Lo:wm}.  

For an example of our results, let $K$ be a complete discretely valued field with algebraically closed characteristic $p$ residue field.  If $X$ is a smooth projective curve defined over $K$ that has bad reduction, but has good reduction over a $\ints/p$-extension $L/K$, and if $\mc{X}$ is a smooth model of $X \times_K L$ over $\mc{O}_L$, then the action of $\Gal(L/K)$ on $\mc{X}$ gives rise to a model $\mc{X}'$ of $X$ with wild quotient singularities, which are all weak if $\mc{X}$ has ordinary special fiber (some might be weak even if $\mc{X}$ does not have ordinary special fiber).  Lorenzini conjectured that if $\mc{X}$ is ordinary, then the resolution graph of each singularity of $\mc{X}'$ contains exactly $sp - 1$ vertices between the vertex corresponding to the strict transform of the special fiber of $\mc{X}'$ and the unique node of the graph, where $s$ is the jump in the ramification filtration for $L/K$.  We prove this for weak wild $\ints/p$-arithmetic quotient singularities on $\mc{X}$ individually, regardless of whether or not the special fiber is ordinary (Corollary \ref{Cresolution} and Remark \ref{Rlorenzini}).  Our techniques differ significantly from what has been used before for these types of problems.  In particular, we rely less on global intersection theory, and more on local deformation theory and valuation theory.  These local techniques allow us to obtain information about weak wild arithmetic quotient singularities independent of the global curves where they appear.

Specifically, we first use a deformation-theoretic argument inspired by work of Bertin and M\'{e}zard (\cite{BM:df}) to show that every weak wild arithmetic quotient singularity over $K$ is formally isomorphic to a singularity arising from a normal model of $\proj^1_K$ with irreducible special fiber (Corollary \ref{CallfromP1}).  This has the immediate consequence mentioned above that all these singularities are rational.  

We then investigate normal models of $\proj^1_K$ using \emph{inductive valuations}, also known as \emph{Mac Lane valuations}.  These were introduced over 80 years ago in \cite{Ma:ca}, but as far as we know they were not used to attack problems involving arithmetic surfaces until the thesis \cite{Ru:mc} of R\"{u}th.  Mac Lane valuations on the rational function field $K(x)$ exactly correspond to normal $\mc{O}_K$-models of $\proj^1_K$ with irreducible special fiber, and general normal $\mc{O}_K$-models of $\proj^1_K$ correspond to finite collections of Mac Lane valuations.  Mac Lane valuations are extremely explicit, and we use them to give regularity conditions for normal models of $\proj^1_K$, which should be of independent interest. (This was first used, in one particular example, in \cite{fkw}.) On the other hand, if a weak wild arithmetic quotient singularity is realized on a normal model of $\proj^1_K$ with irreducible special fiber, we exhibit various properties necessarily satisfied by the corresponding Mac Lane valuation (Theorem \ref{Tclassification}).  Combining this all, we obtain our singularity resolutions.

\subsection{Outline}\label{Soutline}
In \S\ref{Sprelims}, we give some basic results on extensions of discrete valuation fields.  In \S\ref{aqs}, we give definitions and background on arithmetic quotient singularities, in particular the weak and wild ones that are the subject of this paper.  In \S\ref{Sdef}, we use deformation theory to prove that weak wild singularities can be realized inside models of $\proj^1$.  In \S\ref{Smaclane}, we introduce Mac Lane valuations and diskoids, which can be viewed as rigid-analytic analogs to disks when one is working over a non-algebraically closed field.  Diskoids give a useful geometric way of thinking about Mac Lane valuations.  In \S\ref{Sclassification}, we use properties of diskoids to classify weak wild arithmetic quotient singularities in terms of Mac Lane valuations.  Finally, in \S\ref{Sresolution}, we show how to resolve singularities coming from certain collections of Mac Lane valuations, and we exhibit the resolution of weak wild arithmetic quotient singularities as a consequence.  The appendix introduces the concept of an \emph{$N$-path}, which is related to continued fractions and is used in \S\ref{Sresolution} for describing the valuations corresponding to our resolutions of singularities.

\subsection{Notation}\label{Snotation}
Throughout the paper $k$ is an algebraically closed field of
characteristic $p$, and $K$ is a complete
discrete valuation field with residue field $k$.  For any finite
extension $L/K$, we write $\pi_L$ for a uniformizer of $L$, and we normalize the valuation $v_L$ on $L$ so that
$v_L(\pi_L) = 1$.  We write $\mc{O}_L$ for the valuation ring of
$L$.  Note that $L/K$ is totally ramified, a fact that we will use implicitly throughout the paper.  We mainly restrict our consideration to arithmetic surfaces over some $\mc{O}_L$.  This restriction is standard and is justified in \S\ref{Saqsdefs}.

\subsection{Acknowledgements}  We thank Xander Faber and Jim Stankewicz for useful conversations, and we thank the referee for insightful comments that have improved the exposition.

\section{Ramification of extensions of local fields}\label{Sprelims}

Recall from \cite[IV]{Se:lf} that if $L/K$ is a $G$-Galois extension, then for $i \geq -1$ we define the \emph{higher ramification groups}  $G_i := \{\sigma \in G \mid v_L(\sigma(\pi_L) - \pi_L) \geq i+1\}$.  Note that we do not use the so-called ``upper numbering filtration" in this paper.  Then $G = G_{-1} = G_0$ is the inertia group of $L/K$, and $G_1$ is the wild inertia group.  We say that $L/K$ is \emph{weakly ramified} if $G_i$ is trivial for all $i \geq 2$.  We say that $s$ is a \emph{jump} in the higher ramification filtration for $L/K$ if $G_s \supsetneq G_{s+1}$.  

The following proposition is a direct consequence of \cite[p.\ 67, Corollary 3]{Se:lf}.

\begin{prop}\label{Pelemabelian}
If 
$k[[z]]/k[[t]]$ is a weakly ramified $G$-extension with $G$ a
$p$-group, then $G$ is elementary abelian.
\end{prop}

\begin{lem}\label{Lbetterrep}
If $L/K$ is a $G$-Galois extension 
%with $G$ a $p$-group
and $L = K(\alpha)$ for some $\alpha \in L$, then there exists $\delta \in
K$ such that $\alpha = \alpha' + \delta$ with $|G| \nmid v_L(\alpha')$ and
either $\delta = 0$ or $v_L(\delta) < v_L(\alpha')$.
\end{lem}

\proof
Write $\alpha = a_0 + a_1 \pi_L + a_2 \pi_L^2 + \cdots +
a_{|G|-1}\pi_L^{|G|-1}$ with all $a_i \in K$.  If $v_L(\alpha) \neq
v_L(a_0)$, take $\delta = 0$.  Otherwise, take $\delta = a_0$.
\Endproof

\begin{lem}\label{Lvaluation}
If $L/K$ is a $G$-Galois extension with
%with $G$ a $p$-group, and
$v_L(\sigma(\pi_L) - \pi_L) = s + 1$, then $v_L(\sigma(x) - x) \geq s +
v_L(x)$ for all $x \in L$.  If $\sigma$ has $p$-power order and $p \nmid v_L(x)$, then $v_L(\sigma(x) - x) = s + v_L(x)$.  
\end{lem}

\proof
Let $v(x) = \nu$, and write $x = u\pi_L^{\nu}$ with $u \in
\mc{O}_L^{\times}$.  Then $\sigma(x) - x = (\sigma(u) - u)\sigma(\pi_L)^{\nu}
+ u(\sigma(\pi_L)^{\nu} - \pi_L^{\nu})$.  Since the residue field $k$
is algebraically closed, we can write $u = a + b$, with $a \in K$ and
$v_L(b) > 0$.  Thus 
$$v_L((\sigma(u) - u)\sigma(\pi_L)^{\nu}) = v_L((\sigma(b) -
b)\sigma(\pi_L)^{\nu}) \geq s + 1 + \nu.$$ 
On the other hand,  $v_L(u(\sigma(\pi_L)^{\nu} - \pi_L^{\nu})) \geq s + \nu$, with equality if $\sigma$ has $p$-power order and $p \nmid \nu$.  The same therefore holds for $\sigma(x) - x$, which proves the lemma.
\Endproof

\begin{prop}\label{Pramprop}
Suppose $L/K$ is a $G$-extension with $G$ an elementary abelian
$p$-group, and that $\alpha \in \mc{O}_L$ is such that $K(\alpha) = L$ and $v_L(\sigma(\alpha) - \alpha)$ is
independent of the choice of nontrivial $\sigma \in G$.  Let $s$ be maximal such
that the higher ramification group $G_s$ is nontrivial.
\begin{enumerate}[(i)]
\item If $p \nmid v_L(\alpha)$, then $s$ is the unique higher
  ramification jump for $L/K$, and $v_L(\sigma(\alpha) - \alpha) = v_L(\alpha) +
  s$.
\item If $p \mid v_L(\alpha)$, then $v_L(\sigma(\alpha) - \alpha) >
  s + v(\alpha)$.
\end{enumerate}
\end{prop}

\proof
If $p \nmid v_L(\alpha)$, then by Lemma \ref{Lvaluation},
$v_L(\sigma(\pi_L) - \pi_L) = v_L(\sigma(\alpha) - \alpha)) -
v_L(\alpha) + 1$.  Thus $v_L(\sigma(\pi_L) - \pi_L)$ does not depend on the choice of nontrivial $\sigma$, so there is only one ramification jump (namely
$s$).  Since $v_L(\sigma(\pi_L) - \pi_L) = s + 1$, part (i) follows.

To prove part (ii), let $\tau \in G_s$ be a nontrivial element, and
let $M = L^{\langle \tau \rangle}$.  Apply Lemma \ref{Lbetterrep} to $L/M$ to write $\alpha = \alpha' + \delta$ with $\delta \in M$ and $p \nmid v_L(\alpha')$.  Note that $v(\alpha) <
v(\alpha')$.  
Now, $$v_L(\tau(\alpha) - \alpha) = v_L(\tau(\alpha') - \alpha') = s +
v_L(\alpha') > s + v_L(\alpha).$$  By assumption, the same is true
after replacing $\tau$ by any nontrivial $\sigma \in G$. 
\Endproof

\begin{prop}\label{Pnotdivisbyp}
If $L/K$ is a $G$-extension with $G$ a $p$-group and char$(K) = p$, then no ramification jump of $L/K$ is divisible by $p$. 
\end{prop}

\proof
By \cite[IV, Proposition 11]{Se:lf}, it suffices to show that the first jump is not divisible by $p$.  Since the (lower numbering) filtration is compatible with taking subgroups, we may assume $G \cong \ints/p$, in which case the result is well-known (see, e.g., \cite[p.\ 72, Ex. 5]{Se:lf}).
\Endproof

\section{Arithmetic quotient singularities} \label{aqs}

In this section we state and discuss the key definitions used in this paper. In particular, we define the notion of \emph{arithmetic quotient singularity} on an arithmetic surface.

\subsection{}

For the convenience of the reader we start by recalling some facts on minimal regular resolution of arithmetic surfaces. Let $S$ be an excellent connected Dedekind scheme. By an \emph{arithmetic surface} over $S$ we mean a normal $S$-curve $\mc{X}\to S$ (so $\mc{X}\to S$ is of finite type and flat of relative dimension $1$). 

By \cite[Theorem 2.2.2]{CES:cf} there exists a proper birational morphism $\pi:\mc{X}^{\rm reg}\to\mc{X}$ such that $\mc{X}^{\rm reg}$ is a regular $S$-curve, and the fibers of $\pi$ do not contain any $-1$-curves (see \cite[Definition 2.2.1]{CES:cf}). Such an $S$-scheme is unique up to unique isomorphism, and every proper birational morphism $\mc{X}'\to\mc{X}$ with a regular $S$-curve $\mc{X}'$ admits a unique factorization through $\pi$. We call $\pi:\mc{X}^{\rm reg}\to\mc{X}$ the \emph{minimal regular resolution} of $\mc{X}$. 

We remark that if $\mc{X}$ is proper over $S$ and has smooth generic fiber $X$ of genus $\geq 1$, $\mc{X}^{\rm reg}$ is the well-known \emph{minimal regular $S$-model} of $X$ (see e.g.\ \cite{Li:ag}, \S 9.3). 

We will mainly use the following variant of the minimal regular resolution. Let $x\in\mc{X}$ be a closed point on an arithmetic surface over $S$. Let $U\subset\mc{X}$ be an open neighborhood of $x$ which does not contain any nonregular points except $x$. We define the \emph{minimal regular resolution of $\mc{X}$ in $x$} to be the morphism $\pi_x:\mc{X}_x\to\mc{X}$ obtained by gluing $\mc{X}-\{x\}$ to the part of $\mc{X}^{\rm reg}$ lying over $U$ (cf. \cite[Definition 2.2.3]{CES:cf}). (Clearly, all points on $\mc{X}_x$ above $x$ are regular, but there may be nonregular points on $\mc{X}_x$ as well.) By \cite[Corollary 2.2.4]{CES:cf}, $\mc{X}_x$ enjoys uniqueness and minimality properties analogous to $\mc{X}^{\rm reg}$. 

\subsection{}

Let us fix an arithmetic surface $\mc{X}\to S$ and a singular point $x\in\mc{X}$. Let $s\in S$ denote the image of $x$. Let $\pi:\mc{X}'\to\mc{X}$ denote the minimal regular resolution in $x$. The fiber $E:=\pi^{-1}(x)$ is called the \emph{exceptional fiber} of the resolution. We endow $E$ with its reduced induced closed subscheme structure. Since $\mc{X}'$ is regular in any point of $E$, $E$ is an effective Cartier divisor on $\mc{X}'$, which we may write as a sum 
\[
    E = \sum_{i=1}^n  E_i,
\]
where the $E_i$ are the irreducible components of $E$. Each $E_i$ is a connected curve, proper over $k(s)$.

Given an arbitrary divisor $D$ on $\mc{X}'$, we can define an \emph{intersection number}
\[
     D\cdot E_i := \deg_{E_i}(\OO_{E_i}(D))\in\ZZ,
\]     
see \cite{Li:rs}, \S 13. 
In particular, we obtain an intersection pairing on divisors on $\mc{X}$ with support on $E$, which is codified by the \emph{intersection matrix}
\[
   M_x := \big(\, E_i\cdot E_j \,\big).
\]   
It is well known that $M_x$ is symmetric and negative definite (\cite[Lemma 14.1]{Li:rs}). It follows that $E_i\cdot E_i <0$. By definition, $E_i\cdot E_j\geq 0$ for $i\neq j$, and $E_i\cdot E_j=0$ if and only if $E_i\cap E_j=\emptyset$. 

Let us denote by $\mc{X}_s:=\mc{X}\times_S \Spec k(s)$ (resp.\ $\mc{X}_s':=\mc{X}'\times_S\Spec k(s)$) the fiber of $\mc{X}$ (resp.\ $\mc{X}'$) over $s\in S$. Then $\mc{X}'_s$ is a Cartier divisor on $\mc{X}'$, and may be written as a formal sum
\[
    \mc{X}'_s = \sum_Z m_Z\cdot Z, 
\]   
where $Z$ are the irreducible components and $m_z\in\ZZ$. We call $m_Z$ the \emph{multiplicity} of the component $Z$. We write $m_i:=m_{E_i}$. 

By \cite[Prop.\ 10.4 (i)]{Li:rs},
\[
   0 = \mc{X}_s'\cdot E_i = \sum_Z m_Z Z\cdot E_i,
\]
which implies
\begin{equation} \label{eq:self_intersection}
   E_i\cdot E_i = -\frac{1}{m_i} \cdot \sum_{Z\neq E_i} m_Z\cdot Z\cdot E_i.
\end{equation}

\begin{defn} \label{def:intersection_graph}
  The \emph{extended intersection graph} of the desingularization of $x\in\mc{X}$ is the weighted undirected graph $G_x'$ defined as follows. The vertices of $G_x'$ are the irreducible components $Z$ of $\mc{X}_s'$ with $Z\cap E\neq \emptyset$. Two components $Y,Z$ are connected by an edge if $Y\neq Z$ and $Y\cap Z\neq\emptyset$, and the weight of this edge is the intersection number $Y\cdot Z$. The \emph{intersection graph} of the singularity $x\in\mc{X}$ is the (weighted) subgraph $G_x$ of $G_x'$ whose vertices are the irreducible components of $E$. 
\end{defn}

\begin{rem} \label{rem:intersection_graph}
\begin{enumerate}[(i)]
\item
  The graphs $G_x$ and $G_x'$ are connected because $E$ is connected.
\item
  The graph $G_x$ does not depend on the morphism $\mc{X}\to S$ (but $G_x'$ does). 
\item
  Formula \eqref{eq:self_intersection} shows that the full intersection matrix $M_x$ is determined by the extended intersection graph $G_x'$ and the multiplicities $m_Z$. 
\item
  Assume that $\mc{X}_s'$ is a reduced normal crossing divisor. Then $Z\cdot E_i = 1$ for $Z\neq E_i$ and $Z\cap E_i \neq\emptyset$. This means that all edges of $G_x'$ have weight $1$. Formula \eqref{eq:self_intersection} simplifies to 
  \begin{equation} \label{eq:self_intersection2}
      E_i\cdot E_i = -\frac{1}{m_i} \sum_{\substack{Z\neq E_i \\ Z\cap E_i\neq\emptyset}} m_Z.
   \end{equation}
\end{enumerate}   
\end{rem}

\subsection{} 

We continue with the previous notation. 
The following proposition states that the structure of the exceptional fiber of the minimal regular resolution in $x$ depends only on the formal neighborhood of $x$ in $\mc{X}$. 

\begin{prop} \label{prop:lg}
  Let $\mc{X}_1,\mc{X}_2$ be two arithmetic surfaces and $x_i\in\mc{X}_i$, $i=1,2$ 
  be closed points. Assume that the complete local rings $\hat{\OO}_{\mc{X}_i,x_i}$ are isomorphic as $\OO_S$-algebras. Then the exceptional fibers of the minimal regular resolutions in $x_1$ and $x_2$ are isomorphic (including their intersection product). 
\end{prop}

\proof
This follows from the argument in \cite[p.156]{Lipman:desingularization}, used in the proof of Remark D of the introduction of \emph{loc.cit.}. Let $x\in\mc{X}$ be a singular point on an excellent normal surface $\mc{X}$. For the proof we may assume that $\mc{X}=\Spec R$ is local. Let $\hat{R}:=\hat{\OO}_{\mc{X},x}$ be the complete local ring. Then there exists a minimal desingularization $\hat{f}:\hat{\mc{X}}'\to\Spec\hat{R}$. By Remark C of \emph{loc.cit.}, $\hat{f}:\hat{\mc{X}}'\to\Spec\hat{R}$ is the blowup of a primary ideal $\hat{I}\lhd\hat{R}$. Let $f:\mc{X}'\to\mc{X}$ be the blowup of $I:=R\cap\hat{I}$. Then $\hat{\mc{X}}'=\mc{X}'\otimes_R\hat{R}$. Now it follows easily that $f:\mc{X}'\to\mc{X}$ is the minimal desingularization of $x$. This proves the proposition. 
\Endproof

We also include the following proposition, the exact statement of which we could not find in the literature (but compare the paragraph before \cite[Lemma 2.1.1]{CES:cf}).
\begin{prop}\label{Prescommutes}
Let $\mc{X} \to S$ be an arithmetic surface and $x \in \mc{X}$ a closed point lying over a closed point $s \in S$.  Then the minimal regular resolution of $\mc{X}$ in $x$ commutes with base change to $\mc{O}_{S, s}$, to $\hat{\mc{O}}_{S, s}$, as well as to the strict henselization $\mc{O}_{S, s}^{\text{sh}}$.
\end{prop}

\proof
The result for $\mc{O}_{S, s}$ and $\hat{\mc{O}}_{S, s}$ is \cite[Theorem 2.2.4]{CES:cf} (and it also follows easily from Proposition \ref{prop:lg}). To prove the result for $A := \mc{O}_{S, s}^{\text{sh}}$, we may assume $S$ is local.  By Proposition \ref{prop:lg}, embedding $\mc{X}$ into a proper arithmetic surface does not change the minimal resolution in $x$, so we may assume that $\mc{X} \to S$ is proper.  By \cite[Lemma 2.1.1]{CES:cf}, a birational morphism $\mc{X}' \to \mc{X}$ is a regular resolution of $\mc{X}$ in $x$ if and only if the same is true of its base change $\mc{X}'_A \to \mc{X}_A$.  It remains to show that if $\mc{X}' \to \mc{X}$ is a regular resolution of $\mc{X}$ in $x$, then an irreducible component $E$ of $\mc{X}'$ not contained in the strict transform of $\mc{X}$ is a $-1$-curve if and only if the same is true for all irreducible components $F$ of the base change $E_A$.  Let $K_A$ and $K$ be canonical divisors of $\mc{X}'_A$ and $\mc{X}'$, respectively.  Since $\mc{X}'_A \to \mc{X}'$ is unramified, $K_A$ is the pullback of $K$.  So $K \cdot E < 0$ if and only if $K_A \cdot F < 0$ for each $F$. Moreover, $E\cdot E<0$ and $F\cdot F<0$ since the intersection matrix is negativ definite. By \cite[Proposition 9.3.10(a)]{Li:ag}, we are done.
\Endproof

\subsection{}\label{Saqsdefs}
From now on we assume that $S=\Spec\OO_K$, where $K$ is a complete discrete valuation ring with algebraically closed residue field $k$ as in \S\ref{Snotation}.  By Proposition \ref{Prescommutes}, this assumption entails no great loss of generality.  In particular, this situation includes all information about resolution of singularities of arithmetic surfaces over rings of integers in global fields.

Let $\mc{X}\to S$ be an arithmetic surface. We let $X:=\mc{X}\otimes K$ denote the generic and $\Xb:=\mc{X}\otimes k$ the special fiber. Note that $X$ is a smooth $K$-curve. If $\mc{X}\to S$ is proper, then $X/K$ is projective (in fact, even $\mc{X} \to S$ is projective) and we call $\mc{X}$ a \emph{model} of $X$. 

Let $L/K$ be a Galois extension, with Galois group $G$. We let $\mc{Y}:=\tilde{\mc{X}}_L$ denote the normalization of $\mc{X}$ in the function field of $X_L:=X\otimes_K L$. Then $\mc{Y}$ is an arithmetic surface over $S$ (or over $\Spec\OO_L$). The map $\mc{Y}\to\mc{X}$ is finite, and $G$ acts on $\mc{Y}$ in such a way that $\mc{X}=\mc{Y}/G$.

\begin{defn} \label{def:aqs}
  Let $\mc{X}\to S$ be as above, and let $x\in\bar{X}$ be a closed point that is singular on $\mc{X}$. 
\begin{enumerate}[(a)]
\item
  We call $x$ an \emph{arithmetic quotient singularity} on $\mc{X}$ if there exists a finite Galois extension $L/K$ such that $\mc{Y}:=\tilde{\mc{X}}_L$ is smooth over $\OO_L$ at one (equivalently all) points above $x$.  We say that $L/K$ \emph{resolves} the singularity $x \in \mc{X}$.
\item
  We say that $x$ is a \emph{strict arithmetic quotient singularity} if there is a finite Galois extension $L/K$ giving rise to $\mc{Y}$ as in (a) such that for one (or for all) points $y\in\mc{Y}$ above $x$ the action of the stabilizer $G_y\subset G$ of $y$ on the special fiber $\bar{Y}$ of $\mc{Y}$ is faithful. We say that $L/K$ \emph{faithfully resolves} the singularity $x\in\mc{X}$. 
\item
  An arithmetic quotient singularity $x$ is called \emph{weak and wild} if it is strict, and if for one (equivalently all) points $y$ above $x$ the stabilizer $G_y$ is a $p$-group whose action on the complete local ring $\hat{\OO}_{\Yb,y}\cong k[[z]]$ is weakly ramified. In particular, $G_y$ is an elementary abelian $p$-group (see Proposition \ref{Pelemabelian}).
\end{enumerate}
\end{defn}  
  
\begin{rem}\label{Raqs}
\begin{enumerate}[(i)]
\item
  An arithmetic quotient singularity need not be strict.  See Remark \ref{Rstrict} for an example.
\item
  Let $L/K$ be a Galois extension which resolves the arithmetic quotient singularity $x\in\mc{X}$, and let $y\in\mc{Y}:=\tilde{\mc{X}}_L$ be a point above $x$. The stabilizer $G_y\subset G$ of $y$ acts naturally on the complete local ring $\hat{\OO}_{\bar{Y},y}\cong k[[z]]$. Let $I_y\subset G_y$ denote the kernel of $G_y\to{\rm Aut}(k[[z]])$. Then the quotient singularity $x\in\mc{X}$ is strict if and only if $I_y$ is a normal subgroup of $G$ and the quotient scheme $\mc{Y}/I_y$ is smooth over $\OO_L^{I_y}$ at the image of $y$. If this is the case, then the subextension $L':=L^{I_y}/K$ faithfully resolves $x$. 
\item  
  It follows from (ii) that if $x\in\mc{X}$ is a strict arithmetic quotient singularity then the Galois extension $L/K$ which faithfully resolves $x$ is unique. This is implicitly used in Part (c) of the definition. 
\item
  In general, $G_y\neq G$. However, the image of $y$ on the arithmetic surface $\mc{Y}/G_y$ over $\Spec\OO_{L^{G_y}}$ is then an arithmetic quotient singularity which is faithfully resolved by the $G_y$-extension $L/L^{G_y}$. Moreover, it is formally isomorphic to the original singularity. So in view of Proposition \ref{prop:lg} we may assume $G=G_y$ if we are only interested in the structure of the minimal resolution. 
\end{enumerate}
\end{rem}

\begin{defn}
  Let $G$ be a finite group. A \emph{strict arithmetic $G$-quotient singularity} is a strict arithmetic quotient singularity $x\in\mc{X}$ as above such that $G\cong\Gal(L/K)$ for the unique finite Galois extension $L/K$ which faithfully resolves $x$ (Remark \ref{Raqs}(iii)).
\end{defn}

\section{Deformation theory}\label{Sdef}

In this section we prove that every weak wild arithmetic quotient singularity is formally isomorphic to one such singularity on an integral model of the projective line $X=\proj^1_K$. 

\subsection{}

Throughout, we fix a finite Galois extension $L/K$ with Galois group $G$ which is an elementary abelian $p$-group. We choose prime elements $\pi_L$ of $\OO_L$
and $\pi_K$ of $\OO_K$. Note that $\OO_K$ and $\OO_L$ have the same residue field $k$ and hence that $G$ acts trivially on $k$.

Set $\hat{R}:=\OO_L[[T]]$ and let $\tilde{\A}:=\Aut_{\OO_K}(\hat{R})$ denote
the group of continuous $\OO_K$-linear automorphisms of $\hat{R}$. Similarly,
let $\A:=\Aut_{\OO_L}(\hat{R})$ denote the subgroup of $\OO_L$-linear
automorphisms. Then we have a short exact sequence
\begin{equation} \label{eq1}
  1 \to \A \to \tilde{\A} \to G \to 1.
\end{equation}
We are interested in sections $\rho:G\to\tilde{\A}$ of \eqref{eq1}, up to
conjugation by an element of $\A$. We write $\Scal$ for the set of
all sections, and $\bar{\Scal}:=\Scal/\A$ for the set of conjugacy classes. 

Let $\rho\in\Scal$ be a fixed section. Then $\rho$ induces a left
action of $G$ on $\A$ by conjugation:
\[
    {}^\sigma a := \rho(\sigma)\circ a\circ\rho(\sigma)^{-1}.
\] 
We write $\A^{\rho}$ for the group $\A$ considered
as a $G$-module, via this action. Then we have a natural bijection
    $\Scal \iso Z^1(G,\A^{\rho})$,
defined as follows. A section $\rho':G\to\tilde{\A}$
is mapped to the cocycle
\[
   \sigma\mapsto a_\sigma:= \rho'(\sigma)\circ\rho(\sigma)^{-1}.
\]
One easily checks that this map descends to a bijection
$\bar{\Scal} \iso H^1(G,\A^\rho)$.
Throughout, we use the notation from \cite[Chapter I.5]{Se:gc}.
Note that $H^1(G,\A^\rho)$ and $\bar{\Scal}$ are pointed sets but not groups. 

\subsection{}\label{Schooserho}

For any integer $n\geq 0$ we set $\hat{R}_n:=(\OO_L/\pi_K^{n+1})[[T]]$ and let
$\A_n$ denote the set of $\OO_L$-linear continuous automorphisms of
$\hat{R}_n$. Then we define $\tilde{\A}_n$ as the pushout of the extension
\eqref{eq1} along the surjective morphism $\A\to\A_n$:
\begin{equation} \label{eq2}
  \xymatrix{
    1 \ar[r] & \A \ar[r]\ar[d] & \tilde{\A} \ar[r]\ar[d] 
                                              & G \ar[r] \ar[d]^{=}  & 1 \\
    1 \ar[r] & \A_n \ar[r]     & \tilde{\A}_n \ar[r]     
                                              & G \ar[r]  & 1.  }
\end{equation}
An element of $\tilde{\A}_n$ is given by a pair $(a,\sigma)$, where
$a\in\Aut_{\OO_K}(\hat{R}_n)$, $\sigma\in G$ such that $a$ acts on $\mc{O}_L \subseteq \hat{R}_n$ via $\sigma$.

Let $\Scal_n$ denote the set of sections of the lower row
of \eqref{eq2}. We have natural maps $\Scal\to\Scal_n$ and $\Scal_{n+1}\to\Scal_n$
such that $\Scal = \varprojlim_n \Scal_n$.
These induce maps $\bar{\Scal}\to\bar{\Scal}_n$ and $\bar{\Scal}_{n+1}\to\bar{\Scal}_n$
such that $\bar{\Scal} = \varprojlim_n \bar{\Scal}_n$.
If we fix a section $\rho_n:G\to\tilde{\A}_n$ then we have 
bijections
\[
       \Scal_n \iso Z^1(G,\A_n^{\rho_n}), \qquad
          \bar{\Scal}_n \iso H^1(G,\A_n^{\rho_n}).
\]

From now on, we let $\rhob:G\to\tilde{\A}_0\in\Scal_0$ be a fixed section.  We
note that $\tilde{\A}_0 = \Aut_k(k[[T]])\times G$ and hence $\rhob$ is simply
a $k$-linear action of $G$ on $k[[T]]$. We let $\Scal^{\rhob}$ denote the set of
elements of $\Scal$ which lift $\rhob$. Similarly, we obtain subsets
$\bar{\Scal}^{\rhob}$, $\Scal_n^{\rhob}$ and $\bar{\Scal}_n^{\rhob}$. 

For a given $n\geq 1$ we consider the short exact sequence 
\begin{equation} \label{eq3}
   1 \to \Theta_n \to \A_n \to \A_{n-1} \to 1.
\end{equation}
It is easy to see that the kernel $\Theta_n$ is an abelian and normal (but
\emph{not} a central) subgroup of $\A_n$. Elements $a\in\Theta_n$ can be
identified with $k$-linear derivations $\theta:k[[T]]\to k[[T]]$ via
\[
   a=a_\theta:\;f \mapsto f + \pi_L^n\theta(\bar{f}).
\]
Here $f\in \hat{R}_n$ and $\bar{f}$ denotes the image of $f$ in $k[[T]]$. 

If $\rho_n:G\to\tilde{\A}_n$ is a section lifting $\rhob$ then we obtain a
short exact sequence of $G$-modules:
\begin{equation} \label{eq4}
   1 \to \Theta_n \to \A_n^{\rho_n} \to \A_{n-1}^{\rho_{n-1}} \to 1.
\end{equation}
Here we denote by $\rho_{n-1}$ the composition of $\rho_n$ with
$\tilde{\A}_n\to\tilde{\A}_{n-1}$. Also, since the $G$-module structure of
$\Theta_n$ depends only on $\rhob$ which is fixed throughout, $\Theta=\Theta_n$
is independent of $n$ via the identification with ${\rm Der}_k(k[[T]])$. Hence
we drop the indices $\rhob$ and $n$. The $G$-action on $\Theta$ is given by
\begin{equation} \label{eq5}
     {}^\sigma \theta := \rhob(\sigma)\circ\theta\circ\rhob(\sigma)^{-1}.
\end{equation}

By \cite[Chapter I.5.6]{Se:gc}, \eqref{eq4} induces a long exact sequence of pointed sets
\begin{equation} \label{eq6}
  1 \to \Theta^G \to (\A_n^{\rho_n})^G \to (\A_{n-1}^{\rho_{n-1}})^G 
    \to H^1(G,\Theta) \xrightarrow{\delta} H^1(G,\A_n^{\rho_n}) \to
                                H^1(G,\A_{n-1}^{\rho_{n-1}}).
\end{equation}
In particular, the abelian group $H^1(G,\Theta)$ acts, via $\delta$,
transitively on the set of equivalence classes of sections
$\rho_n':G\to\tilde{\A}_n$ which lift $\rho_{n-1}$. Note that this action may
not be faithful, because the image of the map $(\A_{n-1}^{\rho_{n-1}})^G \to H^1(G,\Theta)$
may be a nontrivial subgroup. 

\subsection{}

Let us now fix a section $\rho_{n-1}\in \Scal_{n-1}$. We will show that
the obstruction against lifting $\rho_{n-1}$ to a section $\rho_n\in\Scal_n$ is
represented by an element in $H^2(G,\Theta)$. For this we choose a set
theoretic lift $\rho_n:G\to\tilde{\A}_n$ of $\rho_{n-1}$ and define the map
$a:G^2\to\Theta$ by 
\[
    a_{\sigma,\tau} := \rho_n(\sigma)\rho_n(\tau)\rho_n(\sigma\tau)^{-1}.
\]
A tedious but straightforward computation shows that $a$ is a cocycle, i.e.\
that
\[
    {}^\sigma a_{\tau,\eta}\, a_{\sigma\tau,\eta}^{-1}\,a_{\sigma,\tau\eta}\,
    a_{\sigma,\tau}^{-1} = 1,
\]
for all $\sigma,\tau,\eta\in G$. Let $\Delta(\rho_{n-1})\in H^2(G,\Theta)$
denote the class of $a$. We claim that $\Delta(\rho_{n-1})$ does not depend on
the chosen lift $\rho_n$. Indeed, if $\rho_n'$ is any other lift, we set 
\[
     b_\sigma := \rho_n'(\sigma)\rho_n(\sigma)^{-1},
\]
and then 
\[\begin{split}
   a_{\sigma,\tau}' &:= \rho_n'(\sigma)\rho_n'(\tau)\rho_n'(\sigma\tau)^{-1} \\
     & = b_\sigma\,\rho_n(\sigma)\, b_\tau\, \rho_n(\tau)\,
          \rho_n(\sigma\tau)^{-1}\,b_{\sigma\tau}^{-1}  \\
     & = b_\sigma\,{}^\sigma b_\tau\,\rho_n(\sigma)\,\rho_n(\tau)
          \rho_n(\sigma\tau)^{-1}\,b_{\sigma\tau}^{-1}  \\
     & = \big( b_\sigma\,{}^\sigma b_\tau\,b_{\sigma\tau}^{-1}\big)
     a_{\sigma,\tau},
\end{split}\]
which shows that the cocycles $a$ and $a'$ differ by a coboundary. Now it
follows from the definition that the class $\Delta(\rho_{n-1})\in
H^2(G,\Theta)$ is trivial if and only if there exists a section
$\rho_n\in\Scal_n$ lifting $\rho_{n-1}$.

\subsection{}

Let $R := \hat{R}\cap L(T) \subseteq L((T))$
and let $\B$ denote the group of $\OO_L$-linear automorphisms of $R$. Then $\B$ is
a subgroup of $\A$. We may also consider
$\B$ as a subgroup of $\PGL_2(\OO_L)$, namely
\[
   \B = \{\begin{pmatrix} a & b \\ c & d \end{pmatrix} 
          \mid b \equiv 0\pmod{\pi_L} \}.
\]
Similarly, set $\tilde{\B} = \Aut_{\OO_K}(R)$. Then we have again a short
exact sequence
\begin{equation} \label{eq2.1}
  1 \to \B \to\tilde{\B} \to G \to 1,
\end{equation}
which is the pullback of the sequence \eqref{eq1} via the inclusion
$\B\hookrightarrow\A$. Also, for every $n\geq 0$ we have quotient groups
$\B\to\B_n$, $\tilde{\B}\to\tilde{\B}_n$ which are subgroups of $\A_n$ and
$\tilde{\A}_n$, respectively, which form short exact sequences 
\[
    1 \to \B_n \to\tilde{\B}_n \to G \to 1.
\]
For $n\geq 1$ we let $\LL_n$ denote the kernel of the morphism
$\B_n\to\B_{n-1}$. Elements of $\LL_n$ can be represented by matrices of the
form
\[
     1 + \pi_L^n \begin{pmatrix} a & b \\ c & d \end{pmatrix} ,
\]
with $a,b,c,d\in k$ arbitrary.  Thus we have an identification 
\[
   \LL_n\cong \LL := {\rm Lie}\,\PGL_2(k) = M_{2,2}(k)/<E_2>
\]
(which depends on the choice of $\pi_L$).
In particular, $\LL=\LL_n$ is a $k$-vector space of dimension $3$. The
inclusion $\B_n\hookrightarrow\A_n$ induces an inclusion
$\LL\hookrightarrow\Theta$, and this map identifies $\LL$ with the space of
global sections of the tangent sheaf on $\proj^1_k$. Thus,
\[
    \LL = < \frac{d}{dT}, T\frac{d}{dT}, T^2\frac{d}{dT}>.
\]

\subsection{} \label{defo5}

From now on, we assume that the section $\rhob: G \to \tilde{\A}_0 \in\Scal_0$ chosen in \S\ref{Schooserho} has image in $\B_0 =
\{\begin{pmatrix} a & b \\ c & d \end{pmatrix}
       \in\PGL_2(k) \mid b=0\} \subset \tilde{\A}_0.$
Then $\LL\subset\Theta$ is stable under the $G$-action induced by $\rhob$. 

\begin{thm} \label{thm:defo}
  Assume that the map
      $H^i(G,\LL) \to H^i(G,\Theta)$ is surjective for $i=1$ and 
  injective for $i=2$. Then every section $\rho:G\to\tilde{\A}$ lifting
  $\rhob$ is conjugate to a section $\rho':G\to\tilde{\B}\subset\tilde{\A}$
  lifting $\rhob$.
\end{thm}

\begin{proof}
We prove by induction that for every $n\geq 0$, a section
$\rho_n\in\Scal_n^{\rhob}$ is conjugate to a section $\rho_n':G\to\tilde{\B}_n$
lifting $\rhob$. For $n=0$ the claim is empty, so we may assume $n\geq 1$. We
may also assume, by induction, that the reduction $\rho_{n-1}$ of $\rho_n$ has
image in $\tilde{\B}_{n-1}$. Let $\rho_n':G\to\tilde{\B}_n$ be a set theoretic
lift of $\rho_{n-1}$. Then 
\[
    a_{\sigma,\tau} := \rho_n'(\sigma)\rho_n'(\tau)\rho_n'(\sigma\tau)^{-1}\in \LL
\]
defines a cocycle $a\in Z^2(G,\LL)$ whose class in $H^2(G,\Theta)$ vanishes,
because of the existence of the lift $\rho_n$. Using our assumption on $H^2$
we conclude that the class of $a$ in $H^2(G,\LL)$ is trivial. Hence we may
assume that $\rho_n':G\to\tilde{\B}$ is a group homomorphism. 

Set
\[
    a_\sigma':=\rho_n'(\sigma)\rho_n(\sigma)^{-1}.
\]
Then $a'\in Z^1(G,\Theta)$ is a coboundary. The assumption that $H^1(G,\LL)\to
H^1(G,\Theta)$ is surjective implies that there exists $b\in\Theta$ such that $a_\sigma'':=a_\sigma'\,{}^\sigma b\,b^{-1}$ 
defines a cocycle in $Z^1(G,\LL)$. But then
\[\begin{split}
  \rho_n''(\sigma) & :=(a_\sigma'')^{-1}\,\rho_n'(\sigma) \in \tilde{\B}_n \\
     & = b\,{}^\sigma b^{-1}\,(a_\sigma')^{-1}\,a_\sigma'\,\rho_n(\sigma) \\
     & = b\,{}^\sigma b^{-1}\,\rho_n(\sigma) = b\,\rho_n(\sigma)b^{-1}
\end{split}\]
defines a section $\rho_n'':G\to\tilde{\B}_n$ lifting $\rho_{n-1}$ and which
is conjugate to $\rho_n$. This completes the proof of the theorem.
\end{proof}

\subsection{}

Let $\hat{X}=\Spec A$ be a formal weak wild arithmetic quotient singularity over $\OO_K$ which is faithfully resolved by the extension $L/K$. Recall that this means the following:
\begin{enumerate}
\item
   $A$ is an integral, noetherian, flat, complete local $\OO_K$-algebra with residue field $k$.
\item
  The integral closure of $A$ in $A\otimes_K L$ is formally smooth over $\OO_L$. Hence $\tilde{A}_L\cong\hat{R}:=\OO_L[[T]]$.
\item
  The induced action of $G$ on $k[[T]]$ is faithful and weakly ramified (i.e.\ the extension $k[[T]]/K[[T]]^G$ is weakly ramified). 
\end{enumerate}

Let $\rho:G\to\hat{\A}$ denote the section corresponding to the action on $\hat{R}$ induced by the above identification. Let $\rhob:G\to\A_0$ denote the induced section. 

\begin{lem}\label{LinB0}
  After a a change of the coordinate $T$ we may assume that $\rhob$ has image in $\B_0$. 
\end{lem}

\proof
By the theorem of Katz--Gabber--Harbater (e.g., \cite[Theorem 2.4]{Ha:mp}), the action $\rhob$ of $G$ on $\Spec k[[T]]$ extends to an action of $G$ on an algebraic curve $\Xb/k$ with full inertia group at one point and no inertia elsewhere, such that $\Xb/G \cong \proj^1$.  After a change of variables, we may assume $T \in k(\Xb)$.  Since the $G$-action is weakly ramified, \cite[IV, Proposition 4]{Se:lf} shows that the ramification divisor has degree $2(|G| - 1)$.  The Riemann-Hurwitz formula then shows that $\Xb \cong \proj^1$ as well.  Thus $G$ acts on $T$ via rational functions, and $\rhob \colon G \hookrightarrow \Aut_k(k(T)) \cong PGL_2(k)$.  Now we use that any subgroup of $PGL_2(k)$ isomorphic to $G$ is conjugate to a subgroup of $\mc{B}_0$, the group of lower triangular matrices.   
\Endproof

From now on, we assume that the conclusion of Lemma \ref{LinB0} holds. We are then in the situation of \S \ref{defo5}. In particular, we can consider $\LL$ as a sub-$G$-module of $\Theta$. 

\begin{prop}\label{PallfromP1}
If $\bar{\rho}: G \to \mc{B}_0$ is a section corresponding to a weakly
ramified action on $k[[T]]$ as above, then the map $H^i(G, \LL) \to H^i(G, \Theta)$ is
surjective for $i = 1$ and injective for $i = 2$.
\end{prop}

\proof
By \cite[Lemma 3.4]{CK:ed}, $\LL$ is a direct summand of $\Theta$ as a $G$-module (in that lemma, $M$ is our $\LL$ and $\mc{O}$ is our $\Theta$).  Thus $H^i(G, \LL) \to H^i(G, \Theta)$ is injective for all $i$, which proves the case $i = 2$.
Furthermore, the argument on the top of \cite[\S3.5]{CK:ed} shows that the inclusion $\LL \to \Theta$ induces an isomorphism on $H^1$.
\Endproof

We recall that the natural action of $\PGL_2(F)$ on $F(x)$ for any field $F$ is a \emph{right action}, i.e., $$x \newmatrix{a}{b}{c}{d} = \frac{ax+b}{cx+d}.$$

\begin{cor}\label{CallfromP1}
Every weak wild arithmetic $G$-quotient singularity over $K$ is formally isomorphic to one coming from a $G$-action on a smooth model $\mc{Y}$ of $\proj^1_L$ with free action on the special fiber except for one point fixed by all of $G$, where $G = \Gal(L/K)$ and the $G$-action on the generic fiber $\proj^1_L$ of $\mc{Y}$ is given purely by the $G$-action on $L$. 
\end{cor}

\proof
Combining Theorem \ref{thm:defo} and Proposition \ref{PallfromP1} shows that every weak wild quotient singularity comes from a semilinear action of $G$ on a smooth model of $\proj^1_L$.  This can be represented by an element of $H^1(G, PGL_2(\mc{O}_L))$, where the $G$-action on $PGL_2(\mc{O}_L)$ comes from the given action on $L$.  By the non-abelian version of Hilbert's Theorem 90 (see, e.g., \cite[X, Proposition 3]{Se:lf}), $H^1(G, PGL_2(\mc{O}_L))$ injects into $H^2(G, L^{\times})$, and $H^2(G, L^{\times})$ is trivial by \cite[Corollary and Example (c) on p.\ 80]{Se:gc}.  Thus the $G$-action on $L(\proj^1_L) = L(x)$ is given by a coboundary, so it has the form $g(x) = xB^gB^{-1}$ where $B \in PGL_2(\mc{O}_L)$ is independent of $g \in G$.  Letting $y = xB^{-1}$, we see that $g(y) = x(B^{g})^{-1}B^gB^{-1} = y$ for all $g \in G$.  Thus $G$ fixes $y$, and if $\mc{Y} \cong \proj^1_{\mc{O}_L}$ is the smooth model of $\proj^1_L$ with coordinate $y$, the $G$-action on $\mc{Y}$ is given purely by the $G$-action on $L$.

Since $G$ is a $p$-group and the special fiber of $\mc{Y}$ is isomorphic to $\proj^1_k$, the action of $G$ on $\mc{Y}$ is either trivial or free with the exception of one point with inertia group $G$.  Only the nontrivial case corresponds to a weak wild quotient singularity.  This finishes the proof of the corollary. 
\Endproof

\begin{defn} A $G$-action as in Corollary \ref{CallfromP1} is called a \emph{purely arithmetic} $G$-action.
\end{defn}

Recall that, according to \cite[Definition 1.1]{Li:rs}, a closed point $x$ of a two-dimensional scheme $\mc{X}$ is a \emph{rational singularity} if $\mc{O}_{\mc{X}, x}$ is normal and if there exists a desingularization $\mc{Z}$ of $\Spec \mc{O}_{\mc{X}, x}$ such that $H^1(\mc{Z}, \mc{O}_{\mc{Z}}) = 0$. Equivalently, for every modification $f:\mc{X}'\to\mc{X}$, the stalk of $R^1f_*\OO_{\mc{X}'}$ at $x$ vanishes (this follows from \cite{Li:rs}, Proposition 1.2).

\begin{cor}\label{Crationalsingularity}
Every weak wild arithmetic quotient singularity is a rational singularity.
\end{cor}

\proof
By Corollary \ref{CallfromP1}, such a singularity can be realized on a quotient of a smooth model of $\proj^1_L$ by an arithmetic action of $\Gal(L/K)$.  That is, the singularity can be realized on a normal model $\mc{X}$ of $\proj^1_K$.  Since $H^i(\proj^1_K, \mc{O}_{\proj^1_K}) = 0$ for all $i \geq 1$, we have $H^i(\mc{X}', \mc{O}_{\mc{X}'}) = H^i(\mc{X}, \mc{O}_{\mc{X}}) = 0$ for all modifications $\pi: \mc{X}' \to \mc{X}$ (this follows, for instance, from \cite[Theorem III.12.11(a)]{Ha:ag}). If $f:\mc{X}'\to\mc{X}$ is any modification, then $f_*\mc{O}_{\mc{X}'} \cong \mc{O}_{\mc{X}}$, since $\mc{X}$ is normal. The sheaf $R^1f_*\OO_{\mc{X}'}$ has support in the finitely many closed points of $\mc{X}$ where $f$ is not an isomorphism. The exact sequence of low-degree terms of the Leray spectral sequence yields
\[
   0 \to H^1(\mc{X},\OO_{\mc{X}}) \to H^1(\mc{X}',\OO_{\mc{X}'})
      \to H^0(\mc{X}, R^1f_*\OO_{\mc{X}'})=\oplus_x \big(R^1 f_{*}\OO_{\mc{X}'}\big)_x\to 0.
\]
As we have seen above, the first two terms vanish, so the third term vanishes as well. Thus, the model $\mc{X}$ has only rational singularities. 
\Endproof

\begin{rem}\label{Rrational}
In the equicharacteristic case, Corollary \ref{Crationalsingularity} recovers \cite[Theorem 4.1]{Lo:wq}.
\end{rem}

\section{Mac Lane's theory of inductive valuations}\label{Smaclane}
We give a brief introduction to the theory of inductive valuations,
which was first developed by Mac Lane in \cite{Ma:ca}.  Our main
reference is \cite{Ru:mc}.  Define a \emph{geometric valuation} of
$K(X)$ to be a discrete valuation that restricts to $v_K$ on $K$ and
whose residue field is a finitely generated extension of $k$ with
transcendence degree $1$.  By \cite[Proposition 3.4]{Ru:mc}, normal
models $\mc{X}$ of $\proj^1_K$ correspond to non-empty finite
collections of geometric valuations, by sending $\mc{X}$ to the
collection of geometric valuations corresponding to the local rings at
the generic points of the irreducible components of the special fiber
of $\mc{X}$, given the reduced induced subscheme structure.  

We place a partial order $\preceq$ on valuations by defining $v
\preceq w$ if $v(f) \leq w(f)$ for all $f \in K[x]$.  Let $v_0$ be the
\emph{Gauss valuation} on $K(x)$.  This is defined on $K[x]$ by
$v_0(a_0 + a_1x + \cdots a_nx^n) = \min_{0 \leq i \leq n}v_K(a_i)$,
and then extended to $K(x)$.  A particularly useful way of encoding
geometric valuations $v$ such that $v \succeq v_0$ is as so-called
\emph{inductive valuations}.  These inductive valuations come from
successive ``augmentations'' of the Gauss valuation.  The idea is that 
each augmentation of a given valuation ``declares'' a certain polynomial to have higher valuation than expected.

More specifically, if $v$ is a geometric valuation such that $v
\succeq v_0$, the concept of a \emph{key polynomial over $v$} is
defined in \cite[Definition 4.7]{Ru:mc}.  By definition, these are monic polynomials in $R[x]$.  In the case $v = v_0$, these are precisely the monic linear polynomials --- see Remark \ref{Rkey} below.  If $\phi \in R[x]$ is a
key polynomial over $v$, then for $\lambda \geq v(\phi)$, 
we define an \emph{augmented valuation} $v' = [v, v'(\phi) = \lambda]$ on $K[x]$ by 
$$v'(a_0 + a_1\phi + \cdots + a_r\phi^r) = \min_{0 \leq i \leq r}
v(a_i) + i\lambda$$ whenever the $a_i \in K[x]$ are polynomials with
degree less than $\deg(\phi)$ (we should think of this as a ``base $\phi$ expansion'').  By
 \cite[Lemmas 4.11, 4.17]{Ru:mc}, $v'$ is in fact a discrete
 valuation.  It extends to $K(x)$.  

\begin{rem}
In \cite{Ru:mc}, it is required that $\lambda > v(\phi)$.  Indeed, if $\lambda = v(\phi)$, then $[v, v'(\phi) = \lambda]$ is the same as $v$, but it will sometimes be convenient for us to allow these ``trivially augmented'' valuations.
\end{rem}

We extend this notation to write inductive valuations $$[v_0, v_1(g_1(x)) = \lambda_1, \ldots, v_n(g_n(x)) = \lambda_n]$$ where each $g_i(x) \in R[x]$ is a key polynomial over $v_{i-1}$, we have $\deg g_i(x) \geq \deg g_{i-1}(x)$, and each $\lambda_i$ satisfies $\lambda_i \geq v_{i-1}(g_i(x))$ (by abuse of notation we identify $v_{i-1}$ with $[v_0, v_1(g_1(x)) = \lambda_1, \ldots, v_{i-1}(g_{i-1}(x)) = \lambda_{i-1}]$).  It turns out that set of inductive valuations on $K(x)$ exactly coincides with the set of geometric valuations $v \succeq v_0$ (\cite[Theorem 4.31]{Ru:mc}). Furthermore, every inductive valuation is equal to one where the degrees of the $g_i$ are strictly increasing (\cite[Remark 4.16]{Ru:mc}), so we may and do assume this to be the case for the rest of the paper.

\begin{rem}\label{Rkey}
\begin{enumerate}[(i)]
\item In \cite[Lemma 4.8]{Ru:mc}, it is shown that a key polynomial $\phi$ over
$v_0$ is exactly one such that $v_0(\phi) = 0$ and the residue of
$\phi$ modulo $v_0$ is irreducible.  In particular, since $k$ is
algebraically closed in this paper, $\phi$ must be linear.
Conversely, any monic linear polynomial is a key polynomial over $v_0$.

\item More generally, \cite[Lemma 4.19]{Ru:mc} gives a criterion for recognizing a key
polynomial over an inductive valuation $v = [v_0, v_1(g_1(x)) = \lambda_1,
\ldots, v_n(g_n(x)) = \lambda_n]$.  As a result of this criterion, if
$v = [v_0, v_1(x) = r/p^e]$ where $0 < r <
p^e$ for some $e$ and $p \nmid r$, then any degree $p^e$ polynomial $\phi$ with
constant term $a_0$ satisfying $v(\phi) = v_K(a_0) = r$ is a key
polynomial.  Criteria (i) through (iv) of \cite[Lemma
4.19]{Ru:mc} are immediate, and criterion (v) follows from \cite[Lemma
4.27]{Ru:mc}, using $S = p^{-r}$ and the fact that the residue of
$x^{p^e}/p^r$ is a transcendental generator of the residue ring of $v$ over $k$.
\end{enumerate}
\end{rem}

The following lemma contains some basic observations about inductive valuations.

\begin{lem}\label{Lmultiplicity}
Let $v = [v_0, v_1(g_1(x)) = \lambda_1, \ldots, v_n(g_n(x)) = \lambda_n]$ be an inductive valuation on $K(x)$.  Let $\mc{X}$ be a normal model of $\proj^1_K$ containing an irreducible component $\Vb$ corresponding to $v$.

\begin{enumerate}[(i)]
\item $v(g_i(x)) = \lambda_i$ for all $i$.
\item If $\lambda_i = c_i/d_i$ in lowest terms for all $i$, then the
multiplicity of $\Vb$ in $\Xb$ is $\lcm(d_1, \ldots, d_n)$.  
\item The subring of $K(x)$ of functions
  defined at the generic point of $\Vb$ is exactly the valuation ring of $v$.
\item If $n = 1$ with $g_1(x)$ linear and $\lambda_1 \in \ints_{\geq
    0}$, then the normal model of $\proj^1_K$ corresponding to $\{v\}$ is
  smooth.  Conversely, if $\mc{X}$ is a smooth model of $\proj^1_K$
  such that $x$ is generically defined on the special fiber,
  then $\mc{X}$ corresponds to a geometric valuation $[v_0,
  v_1(g_1(x)) = \lambda_1]$ with $g_1(x)$ linear and $\lambda_1 \in
  \ints_{\geq 0}$.
\item If $n \geq 2$, then $\Vb$ has multiplicity greater than $1$.
\end{enumerate}
\end{lem}

\proof
Part (i) is \cite[Lemma 4.22]{Ru:mc}.  For part (ii), the multiplicity of $\Vb$ in $\Xb$ is just $w(\pi_K)$
where $w$ is the renormalization of $v$ to
have $\ints$ as its value group.  This is clearly $\lcm(d_1, \ldots,
d_n)$.  Part (iii) follows immediately from the correspondence between
irreducible components of normal models and inductive valuations.  The
first direction of part (iv) follows because if $g_1(x) = x - a$, then
the model in question is $\text{Proj } R[\pi_K^{\lambda_1}X_0, X_1 - aX_0]$ where $X_1/X_0 = x$, and this is isomorphic to $\proj^1_R$.
The second direction of part (iv) follows because any smooth model
where $x$ is generically defined on the special fiber can be written as $\text{Proj }R[Y_0, Y_1]$ where $x = a+by$ with $y = Y_1/Y_0$ and $a,b \in R$ with $v_K(b) =: 
\lambda_1 \geq 0$.  This corresponds to $\{v\}$,
where $v = [v_0, v_1(x-a) = \lambda_1]$.  Part (v) follows from part (ii) and the second part of \cite[Corollary 4.30]{Ru:mc}.
\Endproof

\subsection{Diskoids}\label{Sdiskoids}
A \emph{(rigid) diskoid} over $K$, as introduced in \cite[\S4.4]{Ru:mc}, is a union of conjugate disks over $K$. More specifically, if $\phi \in K[x]$ is a monic irreducible polynomial and $\lambda \in \rats_{\geq 0}$, we define the diskoid $D(\phi, \lambda)$ to be the set $\{x \in \ol{K} \, | \, v_K(\phi(x)) \geq \lambda\}$, where, by a slight abuse of notation, we write $v_K$ for the unique extension of $v_K$ to $\ol{K}$ (cf.\ \cite[Definition 4.40]{Ru:mc} --- note that our $K$ is already complete with respect to $v_K$ and we will have no need of the case $\lambda = \infty$).  

Given a diskoid $D$, we can form a valuation $v_D$ on $K[x]$, where $v_D(g(x)) = \inf_{y \in D} v_K(g(y))$.  

\begin{prop}[{{\cite[Theorem 4.56]{Ru:mc}}}]\label{Pdiskoidequivalence}
The map $D \mapsto v_D$ above gives a \emph{bijection} between the set of diskoids over $K$ contained in the closed unit disk around $0$ and the set of inductive valuations on $K[x]$, whose inverse is given by
$$v := [v_0, \ldots, v_n(g_n(x)) = \lambda_n] \mapsto D_v := D(g_n(x), \lambda_n).$$
\end{prop}

\begin{rem}
We need the assumption that the diskoid is in the closed unit disk so that the corresponding valuation is non-negative on $x$.  This is not explicitly stated in \cite{Ru:mc}.
\end{rem}

The following application of Proposition \ref{Pdiskoidequivalence} will be useful for classifying singularities.
\begin{lem}\label{Ldiskoiddescent}
Let $L/K$ be a Galois extension of degree $d$, and suppose $\alpha \in \mc{O}_L$ generates $L$ as a field over $K$.  Let $w = [w_0, w_1(x - \alpha) = \lambda]$, where $w_0$ is the Gauss valuation on $L[x]$.  Suppose $v_L(\alpha' - \alpha'') = \lambda$ for all Galois conjugates $\alpha' \neq \alpha''$ of $\alpha$ over $K$.  The restriction of $w$ to $K[x]$, after rescaling so that it extends $v_K$, is the (unique) inductive valuation of the form $v = [v_0, \ldots, v_n(g_n(x)) = \lambda]$, where $v_0$ is the Gauss valuation of $K$ and $g_n$ is the minimal polynomial of $\alpha$.
\end{lem}

\proof
Since the difference of any two Galois conjugates of $\alpha$ has valuation $\lambda$, we have the equivalences (for $\beta \in \ol{K}$)
\begin{eqnarray*}
v_L(\beta - \alpha) \geq \lambda &\Leftrightarrow& v_L\left(\prod_{\alpha' \sim_{\text{Gal\,}} \alpha} (\beta - \alpha')\right) \geq d\lambda \\
&\Leftrightarrow& v_L(g(\beta)) \geq d\lambda \\
&\Leftrightarrow& v_K(g(\beta)) \geq \lambda.
\end{eqnarray*}
Thus the diskoid $D_w = D(x - \alpha, \lambda)$ corresponding to $w$ over $L$ is equal, as a subset of $\ol{K}$, to the diskoid $D(g(x), \lambda)$ over $K$. By the definition of the valuation associated to a diskoid, the rescaled restriction $v$ of $w$ to $K[x]$ corresponds to the diskoid $v_D$, where $D = D(g(x), \lambda)$.  By Proposition \ref{Pdiskoidequivalence}, $v$ has the desired form when written as an inductive valuation.
\Endproof

\section{Classification of weak wild arithmetic quotient singularities}\label{Sclassification}

\begin{lem}\label{Lpreclassification}
Every weak wild arithmetic quotient singularity over $K$ is faithfully resolved by a Galois extension $L/K$ with Galois group $\Gal(L/K) = (\ints/p)^e$ for some $e$.   Furthermore, the singularity is formally isomorphic to the singularity arising from the quotient by the purely arithmetic action of $\Gal(L/K)$ on the smooth model $\mc{Y}$ of $\proj^1_L$ corresponding to a valuation of the form $$w = [w_0,\, w_1(x-\alpha) = \lambda],$$ where $\alpha$ generates $L$ as a field over $K$ and $0 < v_L(\alpha) < \lambda < p^e$. 
\end{lem}

\proof
By Corollary \ref{CallfromP1}, we know that the singularity is formally isomorphic to one coming from
an action of an elementary abelian group $G \cong (\ints/p)^e$ on a smooth model $\mc{Y}$ of $\proj^1_L$, where $K = L^G$, the action on the generic fiber is purely arithmetic, and there is exactly one fixed point on the special fiber.  Let $L(x)$ be the function field of $\proj^1_L$, where
$x$ is fixed by $G$.  Let $w$
be the valuation of $L(x)$ corresponding to the special fiber of the model $\mc{Y}$, and
let $v$ be the restriction of $w$ to $K(x)$.  Then $v$ corresponds to
the special fiber of the model $\mc{X} := \mc{Y}/G$ of $\proj^1_K$.   After a
change of variable defined over $K$, which does not change the
formal isomorphism class of the singularity, we may assume that $w \succ w_0$
where $w_0$ is the Gauss valuation on $L(x)$ with respect to $x$.  Since $\mc{Y}$ is smooth, Lemma \ref{Lmultiplicity}(iv) shows that $w$ can be expressed in the inductive form
$w =[w_0,\, w_1(x - \alpha) = \lambda],$ where $\alpha \in L$ satisfies
$v_L(\alpha) > 0$ and $\lambda > 0$.  The unique fixed point on the special fiber of $\mc{Y}$ under the $G$-action is the specialization of those $x$ such that $v_L(x-\alpha) < \lambda$, so no $x \in L$ that is fixed by any nontrivial element of $G$ can satisfy $v_L(x - \alpha) \geq \lambda$.  
In particular, the closed disk does not contain the point $0$,
and we may thus assume that $\lambda > v_L(\alpha)$.  Furthermore,
$\alpha$ generates $L$ as a field extension of $K$. 

By Lemma \ref{Lbetterrep}, there exists $\delta \in K$ such that replacing $x$ and $\alpha$ with $x + \delta$ and $\alpha + \delta$ implies that $p^e \nmid v_L(\alpha)$.  Lastly, by replacing $x$ and $\alpha$ with $\nu x$ and $\nu \alpha$ for some $\nu \in K$, we may assume $0 < 
v_L(\alpha) < p^e$.  Since neither of the above substitutions changes the formal isomorphism class of the singularity, the proof is complete.
\Endproof

\begin{rem}\label{Rstrict}
Suppose $\Gal(L/K)$ acts on the smooth model $\mc{Y}$ of $\proj^1_L$ given by $w = [w_0,\ w_1(x - \alpha) = \lambda]$ with $0 < v_L(\alpha) < \lambda \leq v_L(\sigma(\alpha) - \alpha)$ for all nontrivial $\sigma \in \Gal(L/K)$.  Then $\Gal(L/K)$ fixes the specialization of $x = \infty$, 
and the action of $\sigma \in \Gal(L/K)$ on the special fiber of $\mc{Y}$ is nontrivial if and only if $\lambda = v_L(\sigma(\alpha) - \alpha)$.  Thus, the resulting arithmetic quotient singularity is strict if and only if the final inequality is an equality for all nontrivial $\sigma$.  
\end{rem}

We are now able to state our first main theorem.

\begin{thm}\label{Tclassification}
Every weak wild arithmetic quotient singularity faithfully resolved by a Galois extension $L/K$ is formally isomorphic to the unique singularity of the model $\mc{X}$ of $\proj^1_K$ with
irreducible special fiber corresponding to a valuation of the form
\begin{equation}\label{Evaluation}
v = [v_0,\, v_1(g_1(x)) = c_1/p^{e_1}, \ldots, v_{n-1}(g_{n-1}(x)) = c_{n-1}/p^{e_{n-1}},\, v_n(g_n(x)) = c_n]
\end{equation}
where
\begin{enumerate}
\item All $c_i$ are positive integers, and $c_1, \ldots, c_{n-1}$ are prime to $p$.  Furthermore, $c_1 < p^{e_1}$.
\item $g_1(x) = x$.
\item $0 < e_1 < e_2 < \cdots < e_{n-1} = \log_p([L:K])$.
\item $\deg g_i(x) = p^{e_{i-1}}$ for $2 \leq i \leq n$.
\item $g_n(x)$ is irreducible over $K$ and any root of $g_n(x)$ generates the extension $L/K$.
\item $c_n \geq r + s$, where $r = c_1p^{\log_p([L:K])-e_1} = v_K(g_n(0))$ and $s$ is the largest ramification jump of $L/K$ for the lower numbering.  Equality holds if and only if $n = 2$ (equivalently, $e_1 = \log_p([L:K])$, in which case $s$ is the unique ramification jump for $L/K$. 
\end{enumerate}
\end{thm}

\proof
Let $e = \log_p([L:K])$.  By Lemma \ref{Lpreclassification}, the valuation $v$ we seek is $1/p^e$ times the restriction of $w = [v_0,\, v_1(x-\alpha) = \lambda]$ to $K[x]$, where $w$ is the valuation on $L[x]$ from Lemma \ref{Lpreclassification}.  So we need only show that $v$ has the properties (i)--(vi) of the theorem.  First, we make some observations about $\alpha$ and $\lambda$. Let $\mc{Y}$ be as in Lemma \ref{Lpreclassification}.  Since the action of $G := \Gal(L/K)$ is faithful on the special fiber of $\mc{Y}$, we must have $v_L(\sigma(y) - y) = \lambda$ for all nontrivial $\sigma \in G$ and all $y \in L$ with $v_L(y - \alpha) = \lambda$.  In particular, this is true whenever $y$ is Galois conjugate to $\alpha$, so we are in the situation of Lemma \ref{Ldiskoiddescent}.  Thus we can write 
$v = [v_0,\, v_1(g_1) = \lambda_1, \ldots, v_{n-1}(g_{n-1}(x)) = \lambda_{n-1},\ v_n(g_n(x)) = c_n]$, where the $g_i$ are key polynomials of strictly increasing degree, where $g_n$ is the minimal polynomial for $\alpha$, and where $c_n = \lambda$.  This proves (v) and that $c_n$ is an integer.  For $1 \leq i \leq n-1$, write $\lambda_i = c_i/p^{e_i}$ where $c_i$ has no $p$-part.  The $c_i$ are integers since the value group of $v$ is $(1/p^e)\ints$. 

Let $r = v_L(\alpha)$.  Then $r = v_K(N(\alpha)) = v_K(g_n(0))$. By Lemma \ref{Lpreclassification} we may assume that $0 < r = w(\alpha) < c_n = w(x - \alpha) < p^e$.  Since $w(\alpha) < w(x - \alpha)$, we have $w(x) = w(\alpha)$, and consequently $v(x^{p^e}) = v(N(\alpha))$.  In other words, $v(x) = r/p^e < 1$.  Since $x$ is a key polynomial over $v_0$ and since setting $v(x) = r/p^e$ uniquely determines the evaluation of $v$ on any linear polynomial, we can set $g_1(x) = x$ and $\lambda_1 = r/p^e$.  This proves (ii), as well as the assertion in (i) about $c_1$ and the fact that $r = c_1p^{e-e_1}$ in (vi).  In particular, this finishes the proof of (i).

We also note that if $p \nmid r$, then Proposition \ref{Pramprop}(i) shows that $L/K$ has a unique ramification jump $s$ and that $\lambda = c_n = r+s$.  If $p \mid r$, then Proposition \ref{Pramprop}(ii) shows that $\lambda = c_n > r + s$, where $s$ is the largest lower ramification jump for $L/K$.  Since $p \nmid r$ exactly when $e = e_1$, this finishes the proof of (vi).

Since $k$ is algebraically closed, \cite[Corollary 4.30]{Ru:mc}\footnote{\cite[Corollary 4.30]{Ru:mc} is incorrect as stated --- $e(v_m|v_{m-1})$ should be $e(v_{m-1}|v_{m-2})$.} shows that for $i \geq 2$, we have $$\deg(g_i)/\deg(g_{i-1}) = \lcm(p^{e_1}, \ldots, p^{e_{i-1}})/\lcm(p^{e_1}, \ldots, p^{e_{i-2}}),$$ when we set $e_0 = 0$.  Parts (iii) and (iv), except for the statement that $e_{n-1} = e$, then follow from this statement and induction.  By part (i) and the fact that the $e_i$ are increasing from (iii), the value group of $v$ is $(1/p^{e_{n-1}})\ints$, so $e_{n-1} = e$.  This finishes the proof of (iii).
\Endproof

\begin{cor}\label{Cclassification}
Every weak wild arithmetic $\ints/p$-quotient singularity of an arithmetic surface over $K$ is formally isomorphic to the unique singularity of a model $\mc{X}$ of $\proj^1_K$ with
irreducible special fiber corresponding to a valuation of the form
$$[v_0,\, v_1(x) = r/p,\, v_2(g_2(x)) = r + s]$$
where $0 < r < p$ and $g_2(x)$ is an irreducible degree $p$ polynomial, any of whose roots generates a $\ints/p$-extension $L/K$ with ramification jump $s$ that faithfully resolves the singularity.  
\end{cor}

\proof
This follows from Theorem \ref{Tclassification}, since for a strict $\ints/p$-quotient singularity, part (iii) of the theorem implies that $n = 2$.
\Endproof

\begin{defn}\label{Drs}
We call a weak wild arithmetic $\ints/p$-quotient singularity with invariants $(r, s)$ as in Corollary \ref{Cclassification} a \emph{weak wild singularity of type $(r,s)$}.
\end{defn}

\begin{rem}\label{Rrsexists}
If $0 < r < p$ and $s$ is the ramification jump of some $\ints/p$-extension $L/K$, then there exists a weak wild singularity of type $(r, s)$ over $K$.  Namely, we just take $g_2(x)$ to be the minimal polynomial of a generator $\alpha$ of $L/K$ such that $v_L(\alpha) = r$, and then the model $\mc{X}$ of $\proj^1_K$ with irreducible special fiber corresponding to $[v_0, v_1(x) = r/p, v_2(g_2(x)) = r+s]$ contains the desired singularity.  By Proposition \ref{Pnotdivisbyp}, we have $p \nmid s$.  If $K$ has characteristic $p$, then it is well-known that this is the only restriction on $s$.  If $K$ has characteristic zero and absolute ramification index $e_K$, then $s \leq pe_K/(p-1)$ is the only other restriction (see, e.g., \cite[Lemma 3.3(ii)]{Ob:cw}).
\end{rem}

\begin{rem}
  If $\mc{X}$ is the model of $\mathbb{P}^1_K$ associated to $r,s,g_2$ as in Corollary \ref{Cclassification}, but we only require $s$ to be less than or equal to the ramification jump of $L/K$, then we still get an arithmetic quotient singularity $x\in\mc{X}$ which is resolved by $L/K$. If $s$ is strictly less then the ramification jump then the singularity is not strict (see also Remark \ref{Rstrict}). However, we will see later (Corollary \ref{Cresolution}) that the resolution of $x$ does not depend on this condition, but only on the pair $(r,s)$.   
\end{rem}

\begin{rem}\label{Rwhichinvariants}
In fact, we will see in \S\ref{Sresolution} that the resolution graph of a weak wild arithmetic quotient singularity as in Theorem \ref{Tclassification} depends only on $n$, the $c_i$, and the $e_i$.  It would be interesting to know exactly which $n$, $c_i$ and $e_i$ can give rise to such a singularity for a given field $K$ and group $(\ints/p)^e$. This would allow us to determine all possible resolution graphs.  Remark \ref{Rrsexists} gives the answer when $e = 1$.
\end{rem}

\section{Resolution of weak wild singularities}\label{Sresolution}

\subsection{General resolution results}\label{Sgeneral}
We begin by giving general results about the locations of singularities on models of $\proj^1_K$ with irreducible special fiber.  

\begin{lem}\label{LallareA1}
Let $\mc{X}$ be a normal model of $\proj^1_K$ and let $\ol{V}$ be an irreducible component of the special
fiber $\ol{X}$ of $\mc{X}$ with the reduced subscheme structure.  Then $\ol{V} \cong \proj^1_k$.
\end{lem}

\proof
  Let $\iota: \ol{V} \to \ol{X}$ be the inclusion.  Restriction gives a natural surjection $\mc{O}_{\ol{X}} \to \iota_*(\mc{O}_{\ol{V}})$ of sheaves on $\ol{X}$.  Since $\ol{X}$ has dimension 1, this gives rise to a surjection $H^1(X, \mc{O}_{\ol{X}}) \to H^1(X, \iota_*(\mc{O}_{\ol{V}})) \cong H^1(\ol{V}, \mc{O}_{\ol{V}})$.  Since $\mc X \to \Spec \mc{O}_K$ is flat, $H^1(\ol{X}, \mc{O}_{\ol{X}}) = 0$, and thus the same is true for $\ol{V}$.  Since $\ol{V}$ is an integral curve with arithmetic genus $0$, it must be $\proj^1_k$.

\Endproof

The first part of the next lemma is well known, but we include a proof for completeness.

\begin{lem}\label{Lmultiplicitydeterminesregularity}
Let $\mc{X}$ be a normal flat $\mc{O}_K$-curve with generic fiber $X$ and special fiber $\ol{X}$, let $\xb$ be a point of $\ol{X}$, and let $\mc{X}' = \Spec \mc{O}_{\mc{X}, \xb} \subseteq \mc{X}$.

\begin{enumerate}[(i)]
\item If $\xb$ lies on an irreducible component $\ol{V}$ of $\ol{X}$ with multiplicity $m$, and if there is closed point of $X$ specializing to $\xb$ with residue field $L$ where $[L:K] < m$, then $\xb$ is not a regular point of $\mc{X}$.
\item Assume $X \cong \proj^1_K$.  Let $\ol{V}_1, \ldots, \ol{V}_r$ be
  the irreducible components of $\ol{X}$ containing $\xb$, and let
  $D_1, \ldots, D_r$ be the prime divisors supporting $\ol{V}_1,
  \ldots, \ol{V}_r$.  If any $D_i$ is principal when restricted to $\mc{X}'$, then $\xb$ is a regular point of $\mc{X}$.
\end{enumerate}
\end{lem}

\proof
Since regularity is a local property, we may replace $\mc{X}$ with $\mc{X}'$.\\

Part (i): Assume, for a contradiction, that $\mc{X}'$ is regular.
Then the height $1$ prime divisor supporting the closure of the given
$L$-point is principal.  If $w$ generates the corresponding prime
ideal, then $\mc{O}_{\mc{X}, \xb}/(w) \cong S$, where $S \subseteq
\mc{O}_L$ satisfies $\Frac(S) = L$.  The prime divisor $D$ supporting
the restriction of $\ol{V}$ to $\mc{X}'$ is also height $1$, and thus
principal, say corresponding to an ideal $(t)$.  Since $\ol{V}$ has
multiplicity $m$, we have $(t)^m \supseteq (\pi_K)$.  This containment
becomes an equality after taking quotients by $(w)$.  Thus, the ideal $(\pi_K)$ has an
$m$th root in $S \subseteq \mc{O}_L$.  Since $[L:K] < m$, this is a contradiction.\\

Part (ii):  If $i$ is as in the proposition, then
Lemma \ref{LallareA1} (applied to $\mc{X}$) shows that $D_i$ is smooth, so in particular $\xb$ is smooth on $D_i$.  Since $D_i$ is principal, this means that $\xb$ is regular in $\mc{X}'$.

\Endproof

\begin{lem}\label{Lirredsingularity}
Let $\mc{X}$ be the model of $\proj^1_K$ with irreducible special
fiber corresponding to the inductive valuation
$$v = [v_0, v_1(g_1(x)) = \lambda_1, \ldots, v_n(g_n(x)) = \lambda_n].$$
\begin{enumerate}[(i)]
\item There is a singularity at the specialization of $g_n(x) = 0$ if and only
if $\lambda_n$ does not lie in the additive group generated by $1, \lambda_1,
\ldots, \lambda_{n-1}$.  
\item There is a singularity at the specialization of
$x = \infty$ unless $n=1$ and $\lambda_1$ is an integer, in which case there is no singularity.
\item There is no singularity at any other point.
\end{enumerate}
\end{lem}

\proof
Part (i): Let $N > 0$ be such that $(1/N) \ints$ is the additive group generated by $1, \lambda_1, \ldots, \lambda_{n-1}$.  Since $k$ is algebraically closed, repeated application of \cite[Corollary 4.30]{Ru:mc} shows that $N$ is the degree of $g_n$.  So $g_n(x) = 0$ is a point of degree $N$ on the generic fiber.  

If $\lambda_n \notin (1/N)\ints$, then the multiplicity of the special
fiber is greater than $N$, so the specialization $\zb$ of $g_n(x) = 0$ is
singular by Lemma \ref{Lmultiplicitydeterminesregularity}(i).  If
$\lambda_n \in (1/N)\ints$, then the multiplicity of the special fiber
is $N$.  Furthermore, one can construct a rational function $h = c
g_1(x)^{b_1} \cdots g_{n-1}(x)^{b_{n-1}}$ with $c \in K$ with $v(h) =
1/N$.  Since none of the zeroes of $g_1(x), \ldots, g_{n-1}(x)$ have
the same specialization as the zeroes of $g_n(x)$, the divisor of $h$
is locally (near $\zb$) equal to the
unique prime divisor of the special fiber to which $\zb$ specializes.
Lemma \ref{Lmultiplicitydeterminesregularity}(ii) implies that $\zb$ is regular.
\\

Part (ii): The point $x = \infty$ has degree $1$ over $K$.  If the
multiplicity of the special fiber of $\mc{X}$ is greater than $1$,
then Lemma \ref{Lmultiplicitydeterminesregularity}(i) shows that the
specialization of $x = \infty$ is singular.  By Lemma
\ref{Lmultiplicity}(ii) and (v), the multiplicity of the special fiber is
$1$ if and only if $n = 1$ and $\lambda_1$ is an integer.  By Lemma
\ref{Lmultiplicity}, the specialization of $x = \infty$ is regular.  \\

Part (iii):  Let $\zb$ be a point of the special fiber not covered in
parts (i) or (ii).  Let $N > 0$ be such that the additive group
generated by $1, \lambda_1, \ldots, \lambda_n$ is $(1/N)\ints$.  There
is a rational function $h = c g_1(x)^{b_1} \cdots g_n(x)^{b_n}$ such
that $v(h) = 1/N$.  Since none of the zeroes of $g_1(x), \ldots,
g_n(x)$ specialize to $\zb$, the divisor of $h$ is locally (near
$\zb$) the prime divisor corresponding to the special fiber.  
Lemma \ref{Lmultiplicitydeterminesregularity}(ii) shows that $\zb$ is regular in $\mc{X}$.
\Endproof

We now discuss when an intersection point of two irreducible
components of the special fiber of a model can be singular, in a
particular case necessary for us.

\begin{lem}\label{Lcrossingsingularity}
Let $\mc{X}$ be the model of $\proj^1_K$ whose special fiber has two intersecting
irreducible components corresponding to 
$$v = [v_0,\, v_1(g_1(x)) = \lambda_1, \ldots, v_{n-1}(g_{n-1}(x)) = \lambda_{n-1},\, v_n(g_n(x)) =
\lambda_n]$$ and  
$$w = [v_0, v_1(g_1(x)) = \lambda_1, \ldots, v_{n-1}(g_{n-1}(x)) = \lambda_{n-1},\, v_n(g_n(x)) = \lambda_n'],$$
where $\lambda_n' > \lambda_n$.  There is a singularity at the intersection point
of these two irreducible components if and only if
$$\lambda_n' - \lambda_n > \frac{N}{\lcm(N, c)\lcm(N, c')},$$ where 
\begin{itemize}
\item $\lambda_n$ (resp.\ $\lambda_n'$) = $b/c$ (resp.\ $b'/c'$) in lowest terms, 
\item $N > 0$ and $1/N$ generates the additive group generated by $1, \lambda_1, \ldots, \lambda_{n-1}$.
\end{itemize}
\end{lem}

\proof 
Let $y = \alpha g_1(x)^{a_1} \cdots g_{n-1}(x)^{a_{n-1}}$ be such that
$v(y) = w(y) = 1/N$, with $\alpha \in K$ and the $a_i \in \ints$.  Write $m = \gcd(N, c)$ and 
$m' = \gcd(N, c')$.  Note that $m/Nc = 1/\lcm(N, c)$ generates the
value group of $v$.
Let 
$$h = y^{b'N/m'}/g_n^{c'/m}.$$
Then $$(v(h), w(h)) = ((b'c - bc')/m'c, 0) = ((\lambda_n' - \lambda_n)c'/m', 0).$$ 

If $\lambda_n' - \lambda_n = N/\lcm(N, c)\lcm(N, c') = mm'/Ncc'$, then
the ordered pair above is 
$(1/\lcm(N, c), 0)$. 
 In other words, $h$ cuts out the principal prime divisor $D$ locally
 corresponding to the valuation $v$ near the intersection point.
 Lemma \ref{Lmultiplicitydeterminesregularity}(ii) shows that the intersection point is regular.

Note that it is not possible to have $\lambda_n' - \lambda_n <
N/\lcm(N, c)\lcm(N, c')$, because $v(h)$ would be too small to lie in the value group of $v$.  

If $\lambda_n' - \lambda_n > N/\lcm(N, c)\lcm(N, c')$, then take a shortest $N$-path $\lambda_n' = \beta_0/\gamma_0 > \beta_1/\gamma_1 > \cdots > \beta_r/\gamma_r = \lambda_n$, where the $\beta_r/\gamma_r$ are in lowest terms.  Note that $r \geq 2$.  Consider the normal model $\mc{X}'$ of $\proj^1_K$ whose special fiber has $r+1$ irreducible components $\ol{X}_i$ corresponding to 
$$[v_0,\, v_1(g_1(x)) = \lambda_1, \ldots, v_{n-1}(g_{n-1}(x)) = \lambda_{n-1}, v_n(g_n(x)) = \beta_i/\gamma_i]$$ for $0 \leq i \leq r$.  The model $\mc{X}'$ is a blow up of $\mc{X}$.  We claim that $\mc{X}'$ has no $-1$-components outside of the strict transforms $\ol{X}_0$ and $\ol{X}_r$ of the components of the special fiber of $\mc{X}$.  The claim shows that the intersection point of the two irreducible components of the special fiber of $\mc{X}$ is not regular, because if it were, we would be able to blow down $-1$-components of the special fiber of $\mc{X}'$ one by one, eventually obtaining $\mc{X}$.

To prove the claim, we observe that for $0 < i < r$, the multiplicity $\mu_i$ of $\ol{X}_i$ is $\lcm(N, \gamma_i)$.  By Lemma \ref{Lcantblowdown}, either $\mu_{i-1} \geq \mu_i$ or $\mu_{i+1} \geq \mu_i$.  But, by \eqref{eq:self_intersection2}, the self-intersection of $\ol{X}_i$ is $-1$ if and only if $\mu_i = \mu_{i-1} + \mu_{i+1}$.  Since this is impossible, the claim is proved.
\Endproof

The above characterizations of singularities yield two corollaries
about their explicit resolutions.  These corollaries, and the rest of \S\ref{Sresolution}, depend on the concept of \emph{shortest $N$-path} (Definition \ref{Dshortestpath}).

\begin{cor}\label{Cirredresolution}
Let $\mc{X}$ be the model of $\proj^1_K$ from Lemma
\ref{Lirredsingularity}, and let $N$ be such that the additive group
generated by $1, \lambda_1, \ldots, \lambda_{n-1}$ is $(1/N)\ints$. If $\mc{X}$ has a singularity at the
specialization of $g_n(x) = 0$, 
the minimal resolution of this singularity is the normal model $\mc{X}'$ of
$\proj^1_K$ whose special fiber has irreducible components corresponding to 
$$v_{\lambda} := [v_0,\, v_1(g_1(x)) = \lambda_1, \ldots, v_{n-1}(g_{n-1}(x)) = \lambda_{n-1},\, v_n(g_n(x)) =
\lambda]$$ as $\lambda$ runs through a shortest $N$-path from $\lambda_n'$
to $\lambda_n$, where $\lambda_n'$ is the least rational number
greater than $\lambda_n$ that is in $(1/N)\ints$.
\end{cor}

\proof
By Lemma \ref{Lirredsingularity}, any singularities on the exceptional divisor of $\mc{X}'$
can only appear on intersections of two irreducible components (in
particular, there is no singularity at the specialization of $g_n(x) =
0$ to the irreducible component $\ol{X}_{\lambda'_n}$ corresponding to $\lambda_n'$). By
Lemma \ref{Lcrossingsingularity} and Definition \ref{Dshortestpath},
there are in fact no singularities at these intersection points, but
there will be if any component of the exceptional
divisor other than $\ol{X}_{\lambda'_n}$ is blown down.  Furthermore, our shortest $N$-path does not contain any entries in $(1/N)\ints$
other than $\lambda_n'$, so Lemma \ref{Lirredsingularity} shows that blowing down $\ol{X}_{\lambda'_n}$ also
yields a singularity.  Thus $\mc{X}'$ is the minimal resolution. 
\Endproof

\begin{cor}\label{Ccrossingresolution}
Let $\mc{X}$ be the model of $\proj^1_K$ from Lemma
\ref{Lcrossingsingularity}, and let $N$ be as in Lemma
\ref{Lcrossingsingularity}. 
Let $\mc{X}'$ be the normal model of $\proj^1_K$ whose special fiber has irreducible components
corresponding to $$v_{\lambda} := [v_0,\, v_1(g_1(x)) = \lambda_1, \ldots, v_{n-1}(g_{n-1}(x)) = \lambda_{n-1},\, v_n(g_n(x)) =
\lambda]$$ as $\lambda$ runs through a shortest $N$-path from $\lambda_n'$ to $\lambda_n$.
If $\mc{X}$ has a singularity at the intersection
point of the two irreducible components on the special fiber, then $\mc{X}'$ is a (and by unicity, the)
minimal resolution of this singularity.
\end{cor}

\proof
Again, by Lemma \ref{Lirredsingularity}, any singularities on the exceptional divisor of $\mc{X}'$
can only appear on intersections of two irreducible components.  But Lemma \ref{Lcrossingsingularity} and Definition
\ref{Dshortestpath} show that these points are not singular, and
furthermore that blowing down any irreducible component of the
exceptional divisor yields a singularity.
\Endproof

\begin{rem}\label{Runiquepath}
Corollary \ref{Ccrossingresolution} shows that if an $N$-path from $\lambda_n'$ to $\lambda_n$ exists, it is unique.  In fact, one always exists, see Proposition \ref{Puniquepathexists}.
\end{rem}

\subsection{Weak wild quotient singularities}\label{Sweakwildcase}

From Theorem \ref{Tclassification}, we know that any strict weak wild 
arithmetic quotient singularity over $K$ appears in a normal model of $\proj^1_K$.  We give here a general resolution of singularities of any such model with irreducible special fiber.

\begin{thm}\label{Tresolution}
Let $\mc{X}$ be the normal model of $\proj^1_K$ with irreducible special fiber corresponding to a valuation of the form
$$v = [v_0,\, v_1(g_1(x)) = \lambda_1, \ldots, v_{n-1}(g_{n-1}(x)) = \lambda_{n-1},\, v_n(g_n(x)) = \lambda_n].$$  Write $\lambda_i = c_i/d_i$ in lowest terms for $1 \leq i \leq n$, and write $N_i = \lcm(d_1, \ldots, d_i)$.  Set $\lambda_0 = \lfloor \lambda_1 \rfloor$ and set $N_0 = N_{-1} = 1$.  The minimal regular resolution of $\mc{X}$ is the normal model $\mc{X}'$ of $\proj^1_K$ whose special fiber has irreducible components corresponding to the following valuations:
\begin{itemize}
\item For each $1 \leq i \leq n$, the valuation $v_i$ given as part of the inductive valuation $v$.
\item The valuation $\tilde{v}_0 := [v_0, v_1(x) = \lambda_0]$ (note: $\tilde{v}_0 = v_0$ if $0 < \lambda_1 < 1$).
\item For each $1 \leq i \leq n$, the valuations
$$v_{i,\lambda} = [v_0,\, v_1(g_1(x)) = \lambda_1, \ldots, v_{i-1}(g_{i-1}(x)) = \lambda_{i-1},\, v_{i,\lambda}(g_i(x)) = \lambda],$$ as $\lambda$ ranges through the shortest $N_{i-1}$-path from $\alpha_i$ to $\lambda_i$, where $\alpha_i$ is the least rational number greater than $\lambda_i$ in $(1/N_{i-1})\ints$.
\item For each $0 \leq i \leq n-1$, the valuations
$$w_{i, \lambda} = [v_0,\, v_1(g_1(x)) = \lambda_1, \ldots, v_i(g_i(x)) = \lambda_i,\ v_{i+1, \lambda}(g_{i+1}(x)) = \lambda],$$
as $\lambda$ ranges through the shortest $N_i$-path from $\lambda_{i+1}$ to $(N_i/N_{i-1})\lambda_i$.
\end{itemize}
The dual graph of the exceptional fibers of the minimal resolution $\mc{X}'\to\mc{X}$ is shown in Figure \ref{fig:resolutiongraph}. 
\end{thm}

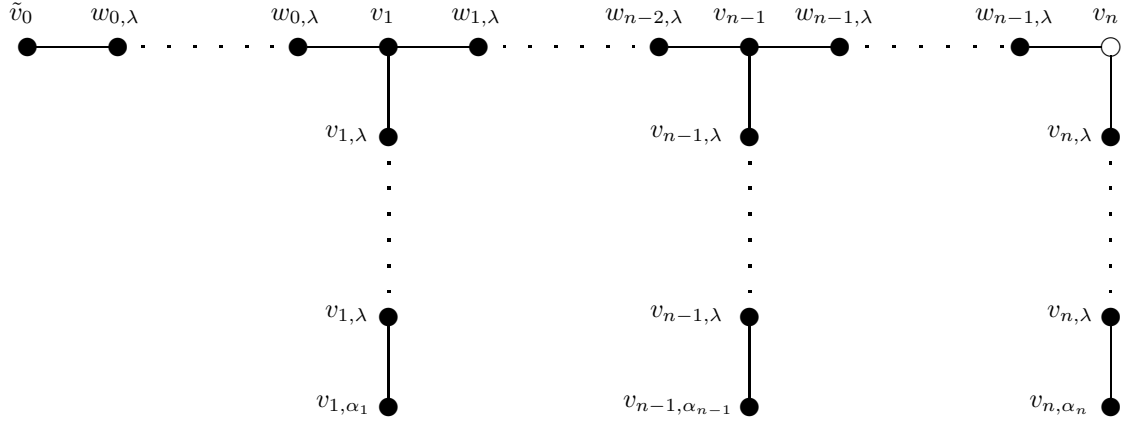
\begin{figure}\label{resolutiondualgraph}
\begin{center}
  \setlength{\unitlength}{1.2mm}
\begin{picture}(140,50)

\put(10,45){\circle*{2}}
\put(8,48){$\tilde{v}_0$}
\put(20,45){\circle*{2}}
\put(17,48){$w_{0,\lambda}$}
\put(40,45){\circle*{2}}
\put(37,48){$w_{0,\lambda}$}
\put(50,45){\circle*{2}}
\put(48,48){$v_1$}
\put(60,45){\circle*{2}}
\put(57,48){$w_{1,\lambda}$}
\put(80,45){\circle*{2}}
\put(74,48){$w_{n-2,\lambda}$}
\put(90,45){\circle*{2}}
\put(86,48){$v_{n-1}$}
\put(130,45){\circle{2}}
\put(128,48){$v_n$}
\put(100,45){\circle*{2}}
\put(95,48){$w_{n-1,\lambda}$}
\put(120,45){\circle*{2}}
\put(115,48){$w_{n-1,\lambda}$}

\put(50,35){\circle*{2}}
\put(43,35){$v_{1,\lambda}$}
\put(50,15){\circle*{2}}
\put(43,15){$v_{1,\lambda}$}
\put(50,5){\circle*{2}}
\put(42,5){$v_{1,\alpha_1}$}
\put(90,35){\circle*{2}}
\put(79,35){$v_{n-1,\lambda}$}
\put(90,15){\circle*{2}}
\put(79,15){$v_{n-1,\lambda}$}
\put(90,5){\circle*{2}}
\put(76,5){$v_{n-1,\alpha_{n-1}}$}
\put(130,35){\circle*{2}}
\put(123,35){$v_{n,\lambda}$}
\put(130,15){\circle*{2}}
\put(123,15){$v_{n,\lambda}$}
\put(130,5){\circle*{2}}
\put(121,5){$v_{n,\alpha_{n}}$}

\put(11,45){\line(1,0){8}}
\put(41,45){\line(1,0){8}}
\put(51,45){\line(1,0){8}}
\put(81,45){\line(1,0){8}}
\put(91,45){\line(1,0){8}}
\put(121,45){\line(1,0){8}}
\put(50,45){\line(0,-1){10}}
\put(90,45){\line(0,-1){10}}
\put(50,15){\line(0,-1){10}}
\put(90,15){\line(0,-1){10}}
\put(130,15){\line(0,-1){10}}
\put(130,44){\line(0,-1){9}}

\linethickness{1pt}
\dottedline{3}(20,45)(40,45)
\dottedline{3}(60,45)(80,45)
\dottedline{3}(100,45)(118,45)
\dottedline{3}(50,35)(50,15)
\dottedline{3}(90,35)(90,15)
\dottedline{3}(130,35)(130,15)

\end{picture}
\end{center}
\caption{The dual graph of the minimal resolution of a normal model of $\proj^1_K$ with irreducible special fiber.  The white vertex corresponds to the strict transform of the special fiber, while the black vertices correspond to components of the exceptional fibers.  The intersection graph of a weak wild quotient singularity corresponds to the complement of the vertices in the right-most vertical column.  The extended intersection graph of the desingularization includes the white vertex as well.} \label{fig:resolutiongraph}
\end{figure}

\proof We first show that $\mc{X}'$ is regular.  Consider the normal model $\mc{X}''$ of $\proj^1_K$ whose special fiber has irreducible components corresponding to the valuations $\tilde{v}_0, v_1, \ldots, v_n$.  Let $\mc{X}_i$ be the blow down of $\mc{X}''$ given by blowing down all irreducible components on the special fiber other than the one corresponding to $v_i$ (or $\tilde{v}_0$, if $i =0$).  By applying Lemma \ref{Lirredsingularity} to each of the $\mc{X}_i$, we see that any singularities of $\mc{X}''$ must lie at the intersection point $\ol{z}_i$ of the two irreducible components of the special fiber corresponding to $v_i$ (or $\tilde{v}_0$ if $i=0$) and $v_{i+1}$ for some $i$, or must lie at the strict transform $\ol{y}_i$ of the specialization of $g_i(x) = 0$ to $\mc{X}_i$ for some $i$.  

By Corollary \ref{Cirredresolution}, the minimal resolution of the singularity at $\ol{y}_i$ on $\mc{X}''$ is the model of $\proj^1_K$ whose special fiber has irreducible components corresponding to $\tilde{v}_0, v_1, \ldots, v_n$, as well as the $v_{i, \lambda}$.  So resolving all of the $\ol{y}_i$ minimally yields a model $\mc{X}'''$ whose special fiber has irreducible components corresponding to $\tilde{v}_0$, the $v_i$ and the $v_{i, \lambda}$.  Write $\ol{z}_i$ again for the strict transform of $\ol{z}_i$ on $\mc{X}'''$.

Let us resolve all of the $\ol{z}_i$.  We note, since $g_{i+1}(x)$ is a key polynomial over $v_i$ (or $\tilde{v}_0$ if $i=0$ and $v_i(g_i(x)) = \lambda_i$, that $v_i(g_{i+1}(x)) = (\deg(g_{i+1})/\deg(g_i)) \lambda_i$ by \cite[Proposition 4.19(iii)]{Ru:mc}.  By \cite[Corollary 4.30]{Ru:mc} (but see the footnote in the proof of Theorem \ref{Tclassification}), $\deg(g_{i+1})/\deg(g_i) = N_i/N_{i-1}$.  
So the valuation $v_i$ can also be written as
$$[v_0, v_1(g_1(x)) = \lambda_1, \ldots, v_i(g_i(x)) = \lambda_i, v_{i+1}(g_{i+1}(x)) = (N_i/N_{i-1})\lambda_i].$$  By Corollary \ref{Ccrossingresolution}, the minimal resolution of the singularities at the $\ol{z}_i$ on $\mc{X}'''$ is the given model $\mc{X}'$, whose special fiber has irreducible components corresponding to $\tilde{v}_0$, the $v_i$, the $v_{i, \lambda}$, and the $w_{i, \lambda}$.  Thus $\mc{X}'$ is regular, and furthermore, it is a minimal regular resolution of $\mc{X}''$.  

It remains to show that $\mc{X}'$ is a minimal regular resolution of $\mc{X}$.  To do this, it suffices to show that the special fiber of $\mc{X}'$ has no $-1$-curves.  Since $\mc{X}'$ is a minimal resolution of $\mc{X}''$, it suffices to check that no strict transform of an irreducible component of the special fiber of $\mc{X}''$ (other than the one corresponding to $v_n$) in $\mc{X}'$ is a $-1$-curve.  Such a component corresponds to $\tilde{v}_0$ or a valuation $v_i$, with $1 \leq i < n$.  By \eqref{eq:self_intersection2}, a component has self-intersection $-1$ if and only if its multiplicity is equal to the sum of the multiplicities of its neighboring components.  

First, suppose $i \geq 1$.  The multiplicity of the irreducible component $\ol{V}$ of the special fiber corresponding to $v_i$ is $N_i$, and the multiplicity of the irreducible component $\ol{W}$ corresponding to $w_{i, \lambda}$ and intersecting $\ol{V}$ is divisible by $N_i$, as can be read off directly from $w_{i, \lambda}$. Since $\ol{W}$ is not the only component of the special fiber of $\mc{X}'$ intersecting $\ol{V}$, it is not possible for $\ol{V}$ to be a $-1$-curve.

Lastly, if $i = 0$, then the multiplicity of the irreducible component $\ol{V}$ corresponding to $\tilde{v}_0$ is $1$, and the multiplicity of the unique neighboring component $w_{0, \lambda}$ is strictly greater than $1$, since $\lambda_0 < \lambda \leq \lambda_1 < \lambda_0 + 1$.  This completes the proof.
\Endproof

\begin{rem}
The special fiber of the resolution $\mc{X}'$ above has simple normal crossings, so $\mc{X}'$ is also the minimal snc-resolution.
\end{rem}

We can now specialize Theorem \ref{Tresolution} to the case of a weak wild arithmetic quotient singularity, which by Theorem \ref{Tclassification}, is isomorphic to the
unique 
singularity of the normal model $\mc{X}$ of $\proj^1_K$ with 
irreducible special fiber corresponding to a valuation of the form in 
(\ref{Evaluation}), where $n \geq 2$ in (\ref{Evaluation}).  By Lemma 
\ref{Lirredsingularity}, the singularity occurs at the specialization 
of $x = \infty$.  

\begin{cor}\label{Cwwsingres}
The dual graph of the minimal resolution of a weak wild quotient singularity is as in Theorem \ref{Tresolution} (Figure \ref{resolutiondualgraph}), except that there are no components corresponding to the $v_{n, \lambda}$. 
\end{cor}

\proof
In light of Theorem \ref{Tclassification}, there is essentially nothing to show.  Since the singularity lies at the specialization of $x = \infty$ on $\mc{X}$, we include only the exceptional divisor of the resolution of this singularity, and thus not the components corresponding to the $v_{n, \lambda}$ (in fact, it is not hard to see that there won't even be any $v_{n, \lambda}$ components in the minimal resolution of $\mc{X}$). 
\Endproof

Recall from Definition \ref{Drs} that a singularity of type $(r,s)$ is a singularity isomorphic to the unique singularity of the model of $\proj^1_K$ with irreducible special fiber corresponding to the valuation
$$v = [v_0, v_1(x) = r/p, v_2(g(x)) = r + s],$$ where $g(x)$ is an irreducible polynomial of degree $p$ giving rise to the associated $\ints/p$-extension $L/K$.  By Corollary \ref{Cclassification}, every weak wild strict arithmetic $\ints/p$-quotient singularity over $K$ has type $(r, s)$ for some $s > 0$ and $0 < r < p$, and $s$ is the ramification jump of $L/K$. 

\begin{defn} \label{def:r-s-graph}
  Let $p$ be a prime number and $r,s$ integers such that $0<r<p$ and $s>0$. Let
     $r/p = [a_0,\ldots,a_k]$
  be the negative continued fraction expansion of $r/p$, with convergents $b_i/c_i$, $i=0,\ldots,k$. Similarly, let
    $p/r = [\tilde{a}_0,\ldots,\tilde{a}_l]$
  be the negative continued fraction expansion of $p/r$, with convergents $\tilde{b}_i/\tilde{c}_i$, $i=0,\ldots,l$. Then the \emph{$(r,s)$-graph} is the extended arithmetic graph depicted in Figure \ref{fig:r-s-graph}, where by extended arithmetic graph we mean the a graph whose vertices are labeled with multiplicity and self-intersection numbers, and one vertex (the ``link'') is drawn in white to represent the strict transform of the component containing the singularity. The graph has a unique node of valency $3$, a unique link and exactly two terminal vertices. 
\end{defn}

\begin{rem} \label{rem:r-s-graph}
\begin{enumerate}[(i)]
\item
  The three vertices of the $(r,s)$-graph adjacent to the unique node have multiplicity $p$, $p-t$ and $t$, where $t$ is the unique integer such that $0<t<p$ and $tr\equiv 1\pmod{p}$. In fact, it follows from Proposition \ref{Pbasiccontfrac} that 
\[
      b_{k-1}p - rc_{k-1} = 1, \quad 0<c_{k-1}<p.
\]
Therefore, $c_{k-1}=p-t$. A similar argument shows that $\tilde{b}_{l-1}=t$. 
\item
  The $(r,s)$-graph is the arithmetic graph defined in \cite[Proposition 4.3]{Lo:wm} (with $t=r(C_1)$). 
\end{enumerate}  
\end{rem}

\begin{figure}
\begin{center}
  \setlength{\unitlength}{1mm}
\begin{picture}(100,40)

\put(20,20){\circle*{2}}
\put(40,20){\circle*{2}}
\put(50,20){\circle*{2}}
\put(60,30){\circle*{2}}
\put(80,30){\circle*{2}}
\put(90,30){\circle*{2}}
\put(60,10){\circle*{2}}
\put(80,10){\circle*{2}}
\put(90,10){\circle*{2}}

\put(10,20){\circle{2}}

\put(11,20){\line(1,0){8}}
\put(41,20){\line(1,0){8}}
\put(50,20){\line(1,1){10}}
\put(81,30){\line(1,0){8}}
\put(50,20){\line(1,-1){10}}
\put(80,10){\line(1,0){10}}

\linethickness{1pt}
\dottedline{3}(20,20)(40,20)
\dottedline{3}(60,10)(80,10)
\dottedline{3}(60,30)(80,30)

\put(9,23){$p$}
\put(19,23){$p$}
\put(39,23){$p$}
\put(49,23){$p$}

\put(58,33){$c_{k-1}$}
\put(78,33){$c_1$}
\put(89,33){$1$}
\put(58,13){$\tilde{b}_{l-1}$}
\put(78,13){$\tilde{b}_{0}$}
\put(89,13){$1$}

\put(16,14.5){$-2$}
\put(36,14.5){$-2$}
\put(46,14.5){$-2$}

\put(57,24.5){$-a_k$}
\put(77,24.5){$-a_2$}
\put(87,24.5){$-a_1$}

\put(57,4.5){$-\tilde{a}_l$}
\put(77,4.5){$-\tilde{a}_1$}
\put(87,4.5){$-\tilde{a}_0$}

\end{picture}

%%% Local Variables: 
%%% mode: latex
%%% TeX-master: "test_r-s-graph"
%%% End: 
\end{center}
\caption{The $(r,s)$-graph, i.e., the extended intersection graph of a singularity of type $(r,s)$.  The labels above the vertices are the multiplicities and the labels below the vertices are the self-intersection numbers.  The total number of $-2$-vertices in the left-hand chain is $sp$.} \label{fig:r-s-graph}
\end{figure}
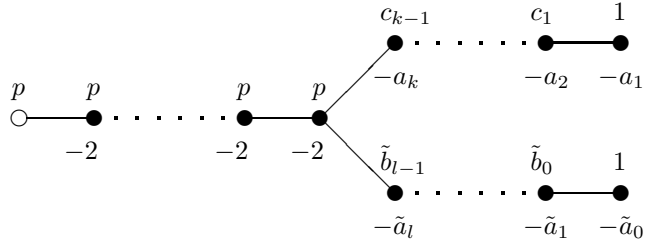

\begin{cor}\label{Cresolution}
\begin{enumerate}[(i)]
\item The resolution graph of the minimal resolution of a weak wild singularity of type $(r,s)$ is the $(r,s)$-graph depicted in Figure \ref{fig:r-s-graph}.
\item If such a singularity is realized in a model $\mc{X}$ of $\proj^1_K$ as in Corollary \ref{Cclassification}, then the valuations corresponding to the irreducible components of the special fiber of the minimal resolution are the Gauss valuation (which corresponds to one of the two terminal vertices of the $(r,s)$-graph), as well as the following valuations, written as inductive valuations:

\begin{itemize}
\item $[v_0,\, v_1(x) = r/p]$; this is the valuation corresponding to the unique node of the resolution graph. 
\item $[v_0,\, v_1(x) = \lambda],$ where either $\lambda$ is a convergent of the negative continued fraction expansion of $r/p$ or $1/\lambda$ is a convergent of the negative continued fraction expansion of $p/r$. In the first case, the corresponding component has multiplicity $c_{i-1}$ and self-intersection $-a_i$, where $1\leq i\leq k$. In the second case, it has multiplicity $\tilde{b}_{i-1}$ and self-intersection $-\tilde{a}_i$, where $0\leq i\leq l$ and $\tilde{b}_{-1}:=1$.
\item $[v_0,\, v_1(x) = r/p,\, v_2(g(x)) = \lambda]$, where $\lambda \in \frac{1}{p}\nats$ and $r < \lambda \leq r + s$. When $\lambda < r+s$, the valuation corresponds to a $-2$-component of the exceptional fiber of multiplicity $p$ and degree $2$.  When $\lambda=r+s$, the valuation corresponds to the strict transform of the special fiber of $\mc{X}$, which is the link of the resolution graph. 
\end{itemize} 
\end{enumerate}
\end{cor}

\proof
The valuations $v_0$, $[v_0,\, v_1(x) = r/p]$, and $[v_0,\, v_1(x) = r/p,\, v_2(g(x)) = r+s]$ are the valuations $v_0$, $v_1$, and $v_2$ from Theorem \ref{Tresolution}.  By Lemma \ref{Lpathtonearestinteger}(i), the convergents of the negative continued fraction expansion of $r/p$ form the shortest $1$-path from $1$ to $r/p$, and thus the valuations $[v_0,\, v_1(x) = \lambda]$ as $\lambda$ runs through these convergents are the valuations $v_{1,\lambda}$ from Theorem \ref{Tresolution}.  By Lemma \ref{Lpathtonearestinteger}(ii), the reciprocals of the convergents of the negative continued fraction expansion of $p/r$ form the shortest $1$-path from $r/p$ to $0$, and thus the valuations $[v_0,\, v_1(x) = \lambda]$ as $1/\lambda$ runs through these convergents are the valuations $v_{0, \lambda}$ from Theorem \ref{Tresolution}.  By definition, the values $\lambda \in (1/p)\ints$ between $r$ and $r+s$ form the shortest $p$-path from $r$ to $r+s$, so the valuations $[v_0,\, v_1(x) = r/p,\, v_2(g(x)) = \lambda]$ are the valuations $w_{1, \lambda}$ from Theorem \ref{Tresolution}.  We have shown that the valuations given in part (ii) are in one-to-one correspondence with the 
valuations from Theorem \ref{Tresolution}, which proves part (ii).

To prove part (i), it remains to compute the self-intersection numbers, which are determined by the intersection graph and the multiplicities.  The multiplicities are calculated using Lemma \ref{Lmultiplicity}(ii), and the self-intersection numbers are calculated using \eqref{eq:self_intersection2}. Using Proposition \ref{Pbasiccontfrac}, it is straightforward to verify that the numbers in Figure \ref{fig:r-s-graph} are correct. 
\Endproof

\begin{cor}\label{Ccontfrac}
The three irreducible components of the special fiber of the minimal
resolution of a singularity of type $(r,s)$ from Corollary \ref{Cresolution} that intersect the component corresponding to $[v_0,\, v_1 = r/p]$ have multiplicity $p$, $t$ and $p-t$ in the special fiber, where $t$ is the unique integer such that $0 < t < p$ and $tr \equiv 1 \pmod{p}$.
\end{cor}

\proof This follows from Remark \ref{rem:r-s-graph}(ii).
\Endproof

\begin{rem}\label{Rlorenzini}
The resolution given in Corollary \ref{Cresolution} is consistent with
the resolution given in \cite[Theorem 6.8]{Lo:wm}.  Furthermore, our
$s$ corresponds to Lorenzini's $\alpha_i/p$ and our
$r$ corresponds to Lorenzini's $r_1(i)^{-1} \pmod{p}$.  In
particular, Corollary \ref{Cresolution} confirms Lorenzini's prediction
about $\alpha_i$ before \cite[Remark 1.1]{Lo:wm}.  Note that our $r$ and $s$ are completely independent from
one another, unlike in the case of a singularity arising from the product of two algebraic curves, as in \cite[Theorem 1.2]{Lo:wq}. 
\end{rem}

Corollary \ref{Cresolution} shows that the result of \cite[Theorem 6.4(b)]{Lo:wm} holds for individual weak wild arithmetic quotient singularities, regardless of whether they come from ordinary curves.  Corollary \ref{Cresolution} also answers the question asked in \cite[Remark 6.9]{Lo:wm} positively.  This has the following consequences, paralleling \cite[Corollaries 6.10, 6.14]{Lo:wm}.

\begin{cor}[cf.\ {\cite[Corollary 6.10]{Lo:wm}}]\label{C610}
Let $X/K$ be a curve with potentially good reduction over a $(\ints/p)^e$-extension $L/K$, such that the natural action of $\Gal(L/K)$ on a good model $X_{\mc{O}_L}$ over $X \times_K L$ gives rise to a weak wild arithmetic quotient singularity.  Then $X(K) \neq \emptyset$.
\end{cor}

\proof
By Theorem \ref{Tresolution}(i), the special fiber of the minimal resolution of $X_{\mc{O}_L}/(\Gal(L/K))$ has an irreducible component with multiplicity $1$, namely, the component corresponding to the Gauss valuation $v_0$.  Since this is a model of $X/K$, there is a point of $X(K)$ that specializes to this component.
\Endproof

\begin{cor}[cf.\ {\cite[Corollary 6.14]{Lo:wm}}, {\cite[Theorem 4.1]{Lo:wqs}})]\label{C614}
Fix any prime $p$ and any $e \geq 1$.  For each integer $m > 0$, there exist a two-dimensional regular local ring $B$ of equicharacteristic $p$ and a two-dimensional regular local ring $B'$ of mixed characteristic, each endowed with an action of $G := (\ints/p)^e$, such that $\Spec B^G$ and $\Spec (B')^G$ are singular exactly at their respective closed points, and the graphs associated with the minimal resolutions of $\Spec B^G$ and $\Spec (B')^G$ have one node and more than $m$ vertices.  
\end{cor}

\proof
In either the equicharacteristic or the mixed characteristic case, one can find $K$ as in our notation and a $G$-extension $L/K$ with arbitrarily high single ramification jump $s$.  Let $g(x)$ be the monic minimal polynomial over $K$ for a uniformizer $\pi_L$ of $L$.  By Theorem \ref{Tresolution}(ii), a weak wild quotient singularity arising from a model $\mc{X}$ of $\proj^1_K$ with irreducible special fiber corresponding to the valuation $[v_0,\, v_1(x) = 1/p^e,\, v_2(g(x)) = 1 + s]$ has resolution graph with one node and at least $ps$ vertices (there are $ps$ vertices represented among the $w_{1, \lambda}$, since the elements of $(1/p^e)\ints$ between $1$ and $1+s$ form a $1/p^e$-shortest path from $1+s$ to $1$).  Such a singularity is a $G$-quotient of a regular two-dimensional local ring, which is our $B$ (or $B'$). 
\Endproof

\appendix

\section{Negative continued fractions and shortest $N$-paths}\label{Sneg}

Given a rational number $y$, there is a unique way of
expressing $y$ in the form 
$$y = a_0 - \frac{1}{a_1 - \frac{1}{\ldots -
    \frac{1}{a_n}}},$$
where the $a_i$ are integers with $a_i \geq 2$ for $i \geq 1$.  This is called the \emph{negative continued fraction expansion} of $y$, and its
truncation $$a_0 - \frac{1}{a_1 - \frac{1}{\ldots -
    \frac{1}{a_i}}},$$
at $a_i$ 
is called the \emph{$i$th convergent}.  For short, we sometimes write such a negative continued fraction expansion as $y = [a_0, a_1, \ldots, a_n]$.

\begin{prop}\label{Pbasiccontfrac}
Let $b_i$ and $c_i$ satisfy the recurrence relations $$b_{i+2} = a_{i+2}b_{i+1} - b_i,\ c_{i+2} = a_{i+2}c_{i+1} - c_i$$ for $i \geq 0$, with $$b_0 = a_0,\ c_0 = 1,\ b_1 = a_0a_1 - 1,\ c_1 = a_1.$$ 
Then the $i$th convergent of $y$ can be written in lowest terms as $b_i/c_i$.  
Furthermore, we have $b_i c_{i+1} - b_{i+1}c_i = 1$ for all $0 \leq i
< n$.  In particular, the convergents form a decreasing sequence.
\end{prop}

\proof
A straightforward proof by induction, see e.g. \cite[Theorem 2.1]{JonesThron}, shows that if $b_i$ and $c_i$ are defined as in the recursion, then the $i$th convergent is $b_i/c_i$ and $b_i c_{i+1} - b_{i+1}c_i = 1$.  This last equality implies that $b_i$ and $c_i$ are relatively prime.
\Endproof

\begin{cor}\label{Cincreasingsequence}
If $b_i/c_i$ is the $i$th convergent of the negative continued
fraction expansion of some $a \in \rats_{>0}$, written in lowest
terms, then the sequence of the $c_i$ is strictly increasing.  The
sequence of the $b_i$ is strictly increasing unless all $b_i$ equal $1$.
\end{cor}

\proof
This follows from the recursive formulas in Proposition \ref{Pbasiccontfrac} using induction and the fact that $a_i \geq 2$ for $i \geq 1$.
\Endproof

\begin{cor}\label{Cconvergentgaps}
If $b_i/c_i$ and $b_j/c_j$ are convergents of the negative continued fraction expansion of $y$ written in lowest terms, and $j \geq i+2$, then $b_i/c_i - b_j/c_j > 1/c_ic_j$.
\end{cor}

\proof
By Corollary \ref{Cincreasingsequence}, the $c_i$ are monotonically increasing.  Thus
$$\frac{b_i}{c_i} - \frac{b_j}{c_j} = \left(\frac{b_i}{c_i} - \frac{b_{i+1}}{c_{i+1}}\right) + \left(\frac{b_{i+1}}{c_{i+1}} - \frac{b_j}{c_j}\right)
= \frac{1}{c_ic_{i+1}} + \left(\frac{b_{i+1}}{c_{i+1}} - \frac{b_j}{c_j}\right) > \frac{1}{c_ic_{i+1}} > \frac{1}{c_ic_j}.$$
\Endproof

\begin{defn}\label{Dshortestpath}
If $a > a' \geq 0$ are rational numbers, and $N$ is a
positive integer, an \emph{$N$-path} from $a$ to $a'$  is a
sequence $a = b_0/c_0 > b_1/c_1 > \cdots > b_n/c_n = a'$ of rational numbers in lowest terms such that
$$\frac{b_i}{c_i} - \frac{b_{i+1}}{c_{i+1}} = \frac{N}{\lcm(N, c_i)\lcm(N, c_{i+1})}$$ for
$0 \leq i \leq n-1$.  If, in addition, no proper subsequence of $b_0/c_0 > \cdots > b_n/c_n$ containing
  $b_0/c_0$ and $b_n/c_n$ is an $N$-path, then the sequence is called
  a \emph{shortest $N$-path} from $a$ to $a'$.
\end{defn}

The following lemma is useful in the proof of Lemma \ref{Lcrossingsingularity}, where the quantities $\lcm(c_i, N)$ are interpreted as multiplicities of irreducible components of the special fiber of an arithmetic surface.

\begin{lem}\label{Lcantblowdown}
If $a > a' > 0$ are rational numbers, and $a = b_0/c_0 > b_1/c_1 > \cdots > b_n/c_n = a'$ is a shortest $N$-path between $a$ and $a'$, then for $0 < i < n$, we have $\lcm(N, c_i) \neq \lcm(N, c_{i-1}) + \lcm(N, c_{i+1})$.  
\end{lem}

\proof
By the definition of an $N$-path,
$$\frac{b_{i-1}}{c_{i-1}} - \frac{b_{i+1}}{c_{i+1}} = \frac{N}{\lcm(N, c_{i-1})\lcm(N, c_i)} + \frac{N}{\lcm(N, c_i)\lcm(N, c_{i+1})} = \frac{N(\lcm(N, c_{i-1}) + \lcm(N, c_{i+1}))}{\lcm(N, c_{i-1})\lcm(N, c_i)\lcm(N, c_{i+1})}.$$

If $\lcm(N, c_i) = \lcm(N, c_{i-1}) + \lcm(N, c_{i+1})$, then the above expression equals $N/(\lcm(N, c_{i-1})\lcm(N, c_{i+1}))$.  But this means that $b_i/c_i$ can be removed while keeping the sequence an $N$-path, contradicting that it is a shortest $N$-path. 
\Endproof

The remainder of this appendix will be devoted to showing that there always exists a unique shortest $N$-path from $a$ to $a'$ (Proposition \ref{Puniquepathexists}).  There will also be a small detour (Corollary \ref{Cshortestpath2} and Example \ref{E3path}) to show how to compute shortest $N$-paths in practice.  

\begin{lem}\label{LdividebyN}
If $b/c \in \rats_{>0}$ is written in lowest terms, and if $b/Nc = \tilde{b}/\tilde{c}$, where $\tilde{b}/\tilde{c}$ is in lowest terms, then $\lcm(N, \tilde{c}) = Nc$. 
\end{lem}

\proof
We have $Nc = \tilde{c} \gcd(N, b)$, 
so $Nc$ is a multiple of both $N$ and $\tilde{c}$.  Since $Nc/N = c$
and $Nc/\tilde{c} = \gcd(N, b)$ are relatively prime (because $b/c$ is in lowest terms), $Nc$ is in fact the least common multiple of $N$ and $\tilde{c}$.
\Endproof

The following lemma allows us to focus on $1$-paths.

\begin{lem}\label{LshortestpathnewN}
The sequence $a = b_0/c_0 > b_1/c_1 > \cdots > b_n/c_n = a'$ is a shortest $1$-path from $a$ to $a'$ if and only if $$\frac{a}{N} = \frac{b_0}{Nc_0} > \cdots > \frac{b_n}{Nc_n} = \frac{a'}{N}$$ is a shortest $N$-path from $a/N$ to $a'/N$ (note that the $b_i/(Nc_i)$ are not necessarily in lowest terms).
\end{lem}

\proof
For $0 \leq i \leq n$, write $b_i/Nc_i$ in lowest terms as $\tilde{b}_i/\tilde{c}_i$.  Since $\tilde{b}_{i-1}/\tilde{c}_{i-1} - \tilde{b}_i/\tilde{c_i} = (1/N)(b_{i-1}/c_{i-1} - b_i/c_i)$, it suffices to show, for $0 \leq i \leq n-1$, that $$\frac{N}{\lcm(N,
  \tilde{c}_i)\lcm(N, \tilde{c}_{i+1})} = \frac{1}{N} \cdot \frac{1}{c_ic_{i+1}}.$$  
Since $\lcm(N, \tilde{c}_i) = Nc_i$ for
all $i$ (Lemma \ref{LdividebyN}), we are done.  
\Endproof

The following proposition depends on Corollary \ref{Ccrossingresolution}, but Corollary \ref{Ccrossingresolution} uses only the definition of shortest $N$-path and Lemma \ref{LshortestpathnewN}, and there is no circular reasoning.
\begin{prop}\label{Punique}
If $a' > a \geq 0$ are rational numbers and $N$ is an integer, a shortest $N$-path from $a$ to $a'$ is necessarily unique.
\end{prop}

\proof
By Lemma \ref{LshortestpathnewN}, it suffices to prove the corollary for $N = 1$.  Consider the normal model $\mc{X}$ of $\proj^1_K$ whose special fiber has two intersecting components corresponding to $v := [v_0,\, v_1(x) = a]$ and $w := [v_0,\, v_1(x) = a']$.  Remark \ref{Runiquepath} applied to this model yields the unicity.
\Endproof

In light of Proposition \ref{Punique}, we will now refer to a shortest $N$-path as \emph{the} shortest $N$-path.
We now turn to existence of $N$-paths.  Again, by Lemma \ref{LshortestpathnewN}, we may assume that $N = 1$.

\begin{lem}\label{Linvertreverse}
Given the shortest $1$-path from $a$ to $a'$, if one takes the reciprocal of all the elements and reverses the order, one obtains the shortest $1$-path from $1/a'$ to $1/a$.
\end{lem}

\proof
Observe that if $b_i/c_i > b_{i+1}/c_{i+1}$ are
two consecutive entries in a $1$-path, written in lowest
terms, then $c_{i+1}/b_{i+1} - c_i/b_i = 1/b_ib_{i+1}$.  This shows
that inverting and reversing a $1$-path yields a $1$-path.  But
since ``inverting and reversing'' is an involution, applying it to the
shortest $1$-path yields the shortest $1$-path.
\Endproof

\begin{lem}\label{Lpathtonearestinteger}
Let $a \in \rats \backslash \ints$ be positive.  
\begin{enumerate}[(i)]
\item The shortest $1$-path from $\lceil a \rceil$ to $a$ is given by the successive convergents in the negative continued fraction for $a$.
\item If $0 < a < 1$, the shortest $1$-path from $a$ to $0$ is given by
taking a shortest $1$-path from $\lceil 1/a \rceil$ to $1/a$, inverting
each entry, reversing the order, and appending $0$ at the end.
In particular, the nonzero entries are the reciprocals of the
convergents of the negative continued fraction expansion of $1/a$.  
\item For general $a$, the shortest $1$-path from $a$ to $\lfloor a
  \rfloor$ is given by adding 
$\lfloor a \rfloor$ to each entry of a shortest $1$-path from $a -
\lfloor a \rfloor$ to $0$, which can be calculated from part (ii).
\end{enumerate}
\end{lem}

\proof
To prove (i), we first note that the first entry in the negative continued fraction expansion
of $a$ is $\lceil a \rceil$.  If $b_i/c_i > b_{i+1}/c_{i+1}$ are two consecutive
convergents written in lowest terms, then Proposition \ref{Pbasiccontfrac} shows that
$b_i/c_i - b_{i+1}/c_{i+1} = 1/c_ic_{i+1}$, so the convergents form a $1$-path.
Given a proper subsequence of convergents, consider two non-consecutive
entries $b_i/c_i > b_{i+r}/c_{i+r}$ with $r > 1$.  By Corollary \ref{Cconvergentgaps}, their difference exceeds $1/c_ic_{i+r}$, so the proper subsequence is not a $1$-path. Thus the convergents in fact form the shortest $1$-path, proving (i).

We now prove (ii).  Lemma \ref{Linvertreverse} shows that the construction in (ii) yields the shortest $1$-path $P$ from
$a$ to $1/\lceil 1/a \rceil$.  Observe that $1/\lceil 1/a \rceil$ is
the only entry in this path with a numerator of $1$, since $\lceil 1/a
\rceil$ is the only integral convergent of the negative continued
fraction expansion of $1/a$.  Appending $0$ ($= 0/1$) at the end of 
$P$ keeps it a $1$-path, and the fact that $1/\lceil 1/a \rceil$ is
the only entry with numerator $1$ shows that it is the shortest
$1$-path.  This proves (ii).  Part (iii) is trivial.
\Endproof

\begin{cor}\label{Cshortestpath1}
For any non-negative rational numbers $a > a'$ such that $\lfloor a \rfloor \geq \lceil a' \rceil$, the shortest $1$-path
from $a$ to $a'$ is given by concatenating paths $P$,
$Q$, and $R$, where $P$ is the shortest $1$-path from
$a$ to $\lfloor a \rfloor$, $Q$ is the $1$-path $\lfloor a \rfloor > 
\lfloor a \rfloor - 1 > \cdots > \lceil a' \rceil + 1 > \lceil a'
\rceil$, and $R$ the shortest $1$-path from $\lceil
a' \rceil$ to $a'$.  In particular, there exists a shortest $1$-path from $a$ to $a'$.
\end{cor}

\proof
By construction, the path $S$ given by concatenating $P$, $Q$, and $R$
is clearly a $1$-path.  It is not possible to remove any element of
$Q$ from $S$ while keeping it a $1$-path, as this would leave two consecutive entries that differ by more
than $1$.  But no entry from the interior of $P$ or $R$ can be removed
either, since $P$ and $R$ are shortest $1$-paths.  So $S$ is the
shortest $1$-path.
\Endproof

\begin{cor}\label{Cshortestpath2}
For any non-negative rational numbers $a > a'$ with $\lfloor Na \rfloor \geq \lceil Na' \rceil$ and any $N > 1$, a shortest $N$-path
from $a$ to $a'$ is given by concatenating paths $P$, $Q$, and $R$, where
\begin{itemize}
\item $P$ is obtained by taking the convergents in the negative
  continued fraction expansion of $1/(Na - \lfloor Na \rfloor)$,
  inverting each convergent, reversing the order, adding $\lfloor Na
  \rfloor$ to each entry, appending $\lfloor Na \rfloor$ at the end, and then
  dividing each entry by $N$. 
\item $Q$ is the path $\lfloor Na \rfloor/N > (\lfloor Na \rfloor -
  1)/N > \cdots > (\lceil Na' \rceil + 1)/N > \lceil Na' \rceil/N$.
\item $R$ is obtained by taking the convergents of the negative
  continued fraction expansion of $Na'$ and dividing each convergent
  by $N$.
\end{itemize}
\end{cor}

\proof
This follows from Corollary \ref{Cshortestpath1}, where path $P$ is constructed using Lemmas
\ref{LshortestpathnewN} and \ref{Lpathtonearestinteger}(iii), path $Q$ is constructed using Lemma \ref{LshortestpathnewN}, and path
$R$ is constructed using Lemmas \ref{LshortestpathnewN} and \ref{Lpathtonearestinteger}(i).
\Endproof

\begin{example}\label{E3path}
A shortest $3$-path from $26/9$ to $2/5$ is given by 
$26/9 > 17/6 > 8/3 > 7/3 > 2 > 5/3 > 4/3 > 1 > 2/3 > 1/2 > 4/9 > 5/12
> 2/5.$  The paths $P$, $Q$, and $R$ from Corollary
\ref{Cshortestpath2} go from $26/9$ to $8/3$, from $8/3$ to $2/3$, and
from $2/3$ to $2/5$, respectively. 
\end{example}

\begin{prop}\label{Puniquepathexists}
For any rational numbers $a > a' \geq 0$ and any positive integer $N$,
there exists a unique shortest $N$-path from $a$ to $a'$.
\end{prop}

\proof
By Lemma \ref{LshortestpathnewN}, we may assume $N = 1$.  Uniqueness follows from Proposition \ref{Punique}.  Let $a = b/c$ and $a' = b'/c'$ in lowest terms.  We use strong induction on $\min(c, c')$.  If either is $1$, then either $a$ or $a'$ is an integer, so $\lfloor a \rfloor \geq \lceil a' \rceil$, and existence follows from Corollary \ref{Cshortestpath1}.  In any case, if there is an integer between $a$ and $a'$, we are done by Corollary \ref{Cshortestpath1}, so assume not.  Subtracting an integer from each entry in a given sequence preserves $1$-paths, and by assumption, there is no integer between $a$ and $a'$ inclusive, so we may assume that $1 > a = b/c > a' = b'/c' > 0$.  By Lemma \ref{Linvertreverse}, it suffices to exhibit a $1$-path from $c'/b'$ to $c/b$.  Since $b < c$ and $b' < c'$, we are done by induction. 
\Endproof

\bibliographystyle{alpha}
\bibliography{main}

\end{document}